\documentclass[a4paper]{amsart}
\usepackage{a4}
\usepackage{amsmath}
\usepackage{amssymb}
\usepackage[applemac]{inputenc}
\usepackage{graphicx}
\usepackage{subcaption}
\usepackage{enumerate}
\usepackage{float}
\usepackage{booktabs}
\usepackage{multirow}
\usepackage{xcolor}
\usepackage{booktabs}

\newcommand{\Tr}{ \mbox{{\rm Tr}}}

\newcommand{\ds}{\displaystyle}
\newcommand{\supp}{\mathrm{supp}\;}

\newcommand{\diver}{{\rm{div}}}

\newtheorem{remark}{\textbf{Remark}}[section]

\newtheorem{lemma}{\textbf{Lemma}}[section]
\newtheorem{theorem}{\textbf{Theorem}}[section]

\newtheorem{proposition}{\textbf{Proposition}}[section]

\numberwithin{equation}{section}

\title[A LG scheme for first order MFGs]{A Lagrange-Galerkin scheme for first order mean field games systems}

\author[Elisabetta Carlini]{Elisabetta Carlini\cs}
\thanks{\cs Dipartimento di Matematica Guido Castelnuovo, Sapienza Universit\`a di Roma, 00185 Rome, Italy (carlini@mat.uniroma1.it).}
\author[Francisco J. Silva]{Francisco J. Silva\za}
\thanks{\za Institut de recherche XLIM-DMI, UMR-CNRS 7252, Facult\'e des Sciences et Techniques, 
Universit\'e de Limoges, 87060 Limoges, France (francisco.silva@unilim.fr).}
\author[Ahmad Zorkot]{Ahmad Zorkot\ha}
\thanks{\ha Institut de recherche XLIM-DMI, UMR-CNRS 7252, Facult\'e des Sciences et Techniques, 
Universit\'e de Limoges, 87060 Limoges, France (ahmad.zorkot@unilim.fr).}
\newcommand{\cs}{$^1$} \newcommand{\za}{$^2$}\newcommand{\ha}{$^3$}

\def\dd{{\rm d}}

\newcommand{\ov}[1]{\overline{#1}}

\def\weight(#1,#2){c_{#1,#2}}

 \if {
  
 } \fi

\if {

}{{\caption}}
} \fi

\def\Dt{\Delta t}
\def\Dx{\Delta x}

\def\A{\mathcal{A}}
\def\B{\mathcal{B}}

\def\G{\mathcal{G}}

\def\I{\mathcal{I}}

\def\L{\mathcal{L}}

\def\P{\mathcal{P}}

\def\SS{\mathcal{S}}

\def\eps{\varepsilon}

\def\supp{\mathop{\rm supp}}

\def\1B{{\bf  1}}

\newcommand{\NN}{\mathbb{N}}
\newcommand{\ZZ}{\mathbb{Z}}

\newcommand{\OO}{\mathcal{O}}
\newcommand{\RR}{\mathbb{R}}

\newcommand\be{\begin{equation}}
\newcommand\ee{\end{equation}}
\newcommand\ba{\begin{array}}
\newcommand\ea{\end{array}}
\newcommand{\bean}{\begin{eqnarray*}}
\newcommand{\eean}{\end{eqnarray*}}

\def\ds{\displaystyle}

\begin{document}

\begin{abstract}
In this work, we consider a first order mean field games system with non-local couplings. A Lagrange-Galerkin scheme for the continuity equation, coupled with a semi-Lagrangian scheme for the Hamilton-Jacobi-Bellman equation, is proposed to discretize the mean field games system. The convergence of solutions to the scheme towards a solution to the mean field game system is established in arbitrary space dimensions. The scheme is implemented to approximate two mean field games systems in dimension one and two.
\end{abstract}

\maketitle
\thispagestyle{empty} 

\vspace{-0.3cm}
{\small
\noindent {\bf AMS-Subject Classification:} 91A16, 49N80, 35Q89, 65M12.  \\[0.5ex]
}
\vspace{-0.4cm}
\section{Introduction}
In view of its applications in Economics, Physics, and Social Sciences, the study of optimal control problems and differential games with a large number of agents has  attracted the attention of several researchers during the last two decades. An important step in this direction has been achieved with the introduction of the theory of Mean Field Games (MFGs) by J.-M. Lasry-Lions \cite{LasryLions06i, LasryLions06ii, LasryLions07} and, independently, by M. Huang, R.P. Malham\'e, and P.E. Caines \cite{HMC06}. The main purpose of this theory is to characterize Nash equilibria for a class of symmetric differential games with a continuum of agents.  One of the main applications of MFGs theory is that such equilibria can be used to provide approximate equilibria for the corresponding games with a large, but finite, number of players. In its standard form, MFGs are described by a system of two Partial Differential Equations (PDEs); a Hamilton-Jacobi-Bellman (HJB) equation, describing the optimal cost of a typical player in the game, and a Fokker-Planck (FP) equation, describing the evolution of the initial distribution when all the players act optimally. We refer the reader to the  monographs~\cite{MR3559742,MR3752669,MR3753660}, the survey~\cite{MR3195844}, and the lectures~\cite{MR4214773} for a throughout overview on MFGs. 

The numerical approximation of MFGs with nonlocal couplings has been an active research topic in recent years. In the case where the MFGs system includes nondegenerate second order terms, finite-difference schemes have been studied in~\cite{AchdouCapuzzo10,MR2974160,MR3097034,MR3452251,MR3698446}, semi-Lagrangian scheme where investigated in~\cite{MR3828859}, and machine learning methods such as deep learning and reinforcement learning have been analyzed  in~\cite{MR4264647,MR4522347,MR4440805}. In the case where the dynamics of the underlying differential games are deterministic, the resulting MFGs system is of first order and several numerical methods have been proposed to approximate its solutions; see e.g.~\cite{MR3148086,camsilva12} for semi-Lagrangian discretizations,~\cite{MR4030259,GS2022} for the approximation by discrete-time finite state space MFGs (see~\cite{gomes10}), and~\cite{MR3962816,MR4322102} for Fourier analysis techniques. We refer the reader to~\cite{MR4214777,MR4368188}, and the references therein, for an overview on numerical methods to approximate MFGs equilibria including also the case of local couplings and variational methods. 

In this paper we focus our attention on the approximation of first order MFGs systems. Namely, we consider the PDE system
\be
\label{MFG} 
\ba{rcl}
-\partial_{t}v+H(x,D_{x}v)&=&F(x, m(t))\quad\hbox{in }]0,T[\times\RR^d,\\[3pt]
v(T,x)&=& G(x,m(T))\quad\mbox{in }\RR^{d},\\[3pt]
\partial_{t}m -\diver\big(D_{p}H(x,D_{x}v)m\big)&=& 0\quad\hbox{in }\RR^d\times ]0,T[,\\[3pt]
m(0)&=& m_0^*,
\ea
\tag{$\text{{\rm MFG}}$}
\ee
\normalsize
where $H\colon\RR^{d}\times\RR^{d}\to\RR$ is convex with respect to its second argument and, denoting by $\P_{1}(\RR^{d})$ the set of probability measures over $\RR^{d}$ with finite first order moment, $F\colon\RR^d\times\P_{1}(\RR^d)\to\RR^d$, $G\colon\RR^d\times\P_{1}(\RR^d)\to\RR^d$, and $m_{0}^{*}\in\P_{1}(\RR^{d})$. In the article~\cite{camsilva12}, the authors propose a convergent semi-discrete scheme to approximate solutions to~\eqref{MFG}. A fully-discrete version has been proposed in~\cite{MR3148086}. In the proposed scheme, the HJB equation is discretized by using a semi-Lagrangian approximation of the HJB equation (see e.g.~\cite{falconeferretilibro}), while the FP equation, or continuity equation, is approximated by a scheme which is dual to a linearized version of the scheme for the HJB equation. The existence of solutions to this approximation is shown and a convergence result to a solution to~\eqref{MFG} is established when the dimension $d$ of the space variable is equal to one. An extension of this scheme to the case where~\eqref{MFG} involves non-local and fractional diffusions terms has been studied in~\cite{Chowdhury_et_al_2022}. If the resulting system has non-smooth solutions, the convergence of solutions of the scheme is also shown when the space dimension is equal to one. 

In order to obtain a convergent scheme for general state dimensions, the key point is to provide a scheme which preserves, under standard conditions on the data (see Section~\ref{sec:assumptions} below), the main properties of solutions to both equations in~\eqref{MFG}. Namely, the boundedness,  Lipschitzianity, and semiconcavity of the solution to the HJB equation and a uniform compact support, equicontinuity, and uniform bounds in $L^{\mathsf{p}}$ spaces for solutions to the continuity equation. As shown in~\cite{MR3180719,MR3148086,Chowdhury_et_al_2022}, a standard semi-Lagrangian scheme for the HJB equation enjoys the former properties under suitable assumptions on the discretization parameters. In order to treat the continuity equation, we consider the Lagrange-Galerkin (LG) scheme introduced in~\cite{morton} and recalled in Section~\ref{sec:lg_continuity_equation} below. As we show, it turns out that, for a specific choice of the basis functions, the resulting scheme for the continuity equation coincides with the one introduced in~\cite{MR2771664} and further studied in~\cite{MR2826759} for Lipschitz velocity fields. The desired  properties for the solutions to this scheme are established in Section~\ref{sec:prop_lg_scheme}. In particular, we provide a uniform $L^{\mathsf{p}}$ estimate, not available in the schemes considered in~\cite{MR3148086,Chowdhury_et_al_2022} in arbitrary space dimensions, which will play a key role in our main convergence result. Combining the semi-Lagrangian scheme for the HJB equation and the LG scheme for the continuity equation, we obtain a discretization of~\eqref{MFG} for which the existence of solutions is established and, using stability and compactness arguments, the convergence to a solution to~\eqref{MFG} is established.

The rest of this article is organized as follows. In Section~\ref{sec:preliminaries} we fix some standard notation, and we state our main assumptions on the data of~\eqref{MFG}. Some important results about solutions to HJB and continuity equations are recalled, as well as existence and uniqueness results for solutions to~\eqref{MFG}. The next two sections deal with the discretization of the HJB and continuity equations in~\eqref{MFG} separately. Section~\ref{sec:sl_hjb} recalls a standard semi-Lagrangian scheme to approximate the solution to  the HJB equation in~\eqref{MFG}. Several important properties of this scheme are reviewed and a new semiconcavity estimate for the solution to the scheme is provided in Proposition~\ref{prop:bound_hessian_semiconcavity_of_v}. This estimate will play a crucial role in Section~\ref{sec:lg_continuity_equation}, which is devoted to the study of a LG scheme to approximate the continuity equation in \eqref{MFG}. Notice that, in general, this continuity equation is driven by a non-smooth velocity field. We show that the solutions to the LG scheme inherit the equicontinuity and $L^{\mathsf{p}}$-stability of the solution to the original equation  and we establish in Proposition~\ref{prop:stability_w_r_t_mu_n} a convergence result as the discretization steps tend to zero. In Section~\ref{sec:LG_scheme_MFG} we couple the schemes studied in the previous sections to obtain a discretization of~\eqref{MFG}. The existence of a solution to the discretized MFG system is provided in Theorem~\ref{th:solution_discrete_mfg} and the convergence result, valid in arbitrary dimensions, is shown in Theorem~\ref{th:convergence_discrete_scheme}. Finally, Section~\ref{sec:num_results} is devoted to the numerical implementation of the scheme for the MFGs system. Since the LG scheme for the continuity equation involves some integrals depending on the discrete characteristics of the equation, we approximate them by numerical quadrature and by the so-called area-weighting method introduced in~\cite{morton}. The performances of these two approximations are compared in a one-dimensional example with an explicit solution, and the area-weighting method is implemented to approximate the solution to a MFGs in a two-dimensional space. 

{\bf Acknowledgements.} F. J. Silva and A. Zorkot where partially supported by l'Agence Nationale de la Recherche (ANR), project  ANR-22-CE40-0010. 
For the purpose of open access, the authors have applied a CC-BY public copyright licence to any Author Accepted Manuscript (AAM) version arising from this submission.

The three authors were partially supported by KAUST through the subaward agreement ORA-2021-CRG10-4674.6.
\section{Preliminaries}
\label{sec:preliminaries}
\subsection{Notation}
Let $d\in\NN$. In what follows, $\langle\cdot,\cdot\rangle$ and $|\cdot|$ denote the standard scalar product in $\RR^d$ and its induced  norm, respectively. We set $|\cdot|_{\infty}$ for the maximum norm in $\RR^d$ and $B_{\infty}(0,C)$ and $\ov{B}_{\infty}(0,C)$ for the associated open and closed balls, centered at $0$ and of radius $C>0$, respectively. Let $\P(\RR^d)$ be the set of probability measures on $\RR^d$. For every $\nu\in\P(\RR^{d})$ we denote by $\supp(\nu)$ its support. Let $\P_1(\RR^d)=\{\nu\in\P(\RR^d)\,|\,\int_{\RR^d}|x|\dd\nu(x)<\infty\}$, and, for every $\nu_1$, $\nu_2\in\P_1(\RR^d)$, set 
\be
d_{1}(\nu_1,\nu_2)=\underset{\gamma\in\Pi(\nu_1,\nu_2)}{\inf}\int_{\RR^d\times\RR^d}|x-y|\dd \gamma(x,y),
\ee
where $\Pi(\nu_1,\nu_2)$ denotes the set of probabilities measures on $\RR^d\times\RR^d$ with first and second marginals given by $\nu_1$ and $\nu_2$, respectively. By the Kantorovich-Rubinstein theorem (see e.g.~\cite[Section~7.1]{Ambrosiogiglisav}) we have
\be
\label{e:d_1_alternative} 
d_1(\nu_1,\nu_2)=\sup\bigg\{\int_{\RR^d}\varphi(x)\dd(\nu_1-\nu_2)(x)\,\big|\,\varphi\in\text{Lip}_1(\RR^d)\bigg\},
\ee
where $\text{Lip}_1(\RR^d)$ denotes the set of all nonexpansive functions on $\RR^d$. Given $\nu\in\P(\RR^d)$ and a Borel function $\Psi:\RR^d\to\RR^q$ ($q\in\NN$), the {\it push-forward} measure $\Psi\sharp\nu$, defined on the $\sigma$-algebra of Borel sets $\B(\RR^{q})$, is defined by
\be
\Psi\sharp\nu(A)=\nu(\Psi^{-1}(A))\quad\text{for all }A\in \B(\RR^q), 
\ee
or, equivalently (see e.g.~\cite[Theorem 3.6.1]{MR2267655}), for every $\varphi\colon\RR^d \to\RR$ such that $\varphi\circ\Psi$ is integrable with respect to $\nu$, one has 
\be
\label{image_measure}  
\int_{\RR^q}\varphi(x)\dd(\Psi\sharp\nu)(x)=\int_{\RR^{d}}\varphi\big(\Psi(x)\big)\dd\nu(x).
\ee 

\subsection{Assumptions}
\label{sec:assumptions}
Our hypothesis on the data of~\eqref{MFG} are the following:
\begin{enumerate}[{\bf(H1)}] 
\item 
\label{h:h1}
It holds that
\be
\label{def:H}
H(x,p)=\sup_{a\in\RR^d}\big(\langle a,p\rangle-L(x,a)\big)\quad\text{for all }x,\,p\in\RR^d,
\ee
where $L\colon\RR^d\times\RR^d\to\RR$ is of class $C^2$, bounded from below, and, for every $x$, $a\in\RR^{d}$, we have
\begin{align}
L(x,a)&\leq C_{L,1}|a|^2+C_{L,2},
\label{h:L_bounded_above_quadratic_term}\\
|D_{x}L(x,a)|&\leq C_{L,3}(1+|a|^2),
\label{h:L_Lipschitz}\\
C_{L,4}|b|^2&\leq \langle D^{2}_{aa}L(x,a)b,b\rangle\quad\text{for all }b\in\RR^d,
\label{h:L_strong_convexity}\\
\langle D^{2}_{xx}L(x,a)y,y\rangle&\leq C_{L,5}(1+|a|^2)|y|^2\quad\text{for all }y\in\RR^d,
\label{h:L_second_order_derivative_x_bounded_above}
\end{align}
for some constants $C_{L,i}>0$ ($i=1,\hdots,5$).
\vspace{0.2cm}
\item 
\label{h:h2}
The functions $F$ and $G$ are continuous and, for every $x$, $y\in\RR^d$ and
$\nu\in\P_1(\RR^d)$, we have
\begin{align}
|F(x,\nu)|&\leq C_{F,1},
\label{h:F_bounded}\\
|G(x,\nu)|&\leq C_{G,1},
\label{h:G_bounded}\\
|F(x,\nu)-F(y,\nu)|&\leq C_{F,2}|x-y|,
\label{h:F_Lipschitz}\\
|G(x,\nu)-G(y,\nu)|&\leq C_{G,2}|x-y|,
\label{h:G_Lipschitz}\\
F(x+y,\nu)-2F(x,\nu)+F(x-y,\nu)&\leq C_{F,3}|y|^{2}, 
\label{h:F_semiconcave}\\
G(x+y,\nu)-2G(x,\nu)+G(x-y,\nu)&\leq C_{G,3}|y|^{2},
\label{h:G_semiconcave}
\end{align}
for some constants $C_{F,i}>0$, $C_{G,i}>0$ ($i=1,2,3$). 
\vspace{0.2cm}
\item
\label{h:initial_condition}
The initial condition $m_{0}^{*}$ is absolutely continuous with respect to the Lebesgue measure and satisfies:
\begin{enumerate}[{\rm(i)}]
\item
\label{h:initial_condition_i} 
There exists $C^*>0$ such that $\supp(m_{0}^{*})\subset\ov{B}_{\infty}(0,C^*)$.
\item
\label{h:initial_condition_ii}
There exists $\mathsf{p}\in]1,\infty]$ such that the density of $m_{0}^{*}$, still denoted by $m_{0}^{*}$, belongs to $L^{\mathsf{p}}(\RR^d)$.
\end{enumerate}
\end{enumerate}
\begin{remark}
\label{rem:H_and_L} 
Since $L$ is bounded from below, the strong convexity assumption~\eqref{h:L_strong_convexity} on $L(x,\cdot)$, which is uniform with respect to $x\in\RR^d$, and~\eqref{h:L_bounded_above_quadratic_term}, imply the existence of $C_{L,6}>0$ and $C_{L,7}>0$ such that
\be
\label{h:L_bounded_below_quadratic_term}
L(x,a)\geq C_{L,6}|a|^{2}-C_{L,7}\quad\text{for all }x,\,a\in\RR^d.
\ee
It follows from~\eqref{def:H},~\eqref{h:L_bounded_above_quadratic_term}, and~\eqref{h:L_bounded_below_quadratic_term}, that there exist $C_{H,i}>0$ {\rm(}$i=1,2,3,4${\rm)} such that 
\be 
\label{h:quadratic_bounds_H}
C_{H,1}|p|^2-C_{H,2}\leq H(x,p)\leq C_{H,3}|p|^2+C_{H,4}\quad\text{for all }x,\,p\in\RR^d.
\ee
Moreover, by~\eqref{def:H},~\eqref{h:L_bounded_below_quadratic_term}, and Danskin's theorem {\rm(}see e.g.~\cite[Theorem 4.13]{BonSha}{\rm)}, we deduce that $H$ is of class $C^{1}$ and, for every $x$, $p\in\RR^d$, the following equalities hold
\begin{align}
D_{a}L(x,D_{p}H(x,p))&=p, \label{e:derivative_p_H}\\
D_{x}H(x,p)&=-D_{x}L(x,D_{p}H(x,p)).\label{e:derivative_x_H}
\end{align}
Since $D_{p}H(x,p)$ is the unique maximizer of $\sup_{a\in\RR^{d}}\big(\langle a,p\rangle-L(x,a)\big)$,~\eqref{h:L_bounded_above_quadratic_term}, and~\eqref{h:L_bounded_below_quadratic_term}, yield the existence of $C_{H,5}>0$ such that 
\begin{equation}
\label{eq:D_P_H_linear_growth}
|D_{p}H(x,p)|\leq C_{H,5}(1+|p|)\quad\text{for all }x,\,p\in\RR^{d}.
\end{equation}
Finally, since $L$ is of class $C^2$, by~\eqref{h:L_strong_convexity} and the implicit function theorem applied to~\eqref{e:derivative_p_H}, it follows that $D_{p}H$ is of class $C^1$ and hence, by~\eqref{e:derivative_x_H}, we obtain that $H$ is of class $C^2$. 

A typical example of a function $H$ satisfying~{\bf(H\ref{h:h1})} is given by
$H(x,p)=a(x)|p|^{2}+\langle b(x),p\rangle$, where $a\colon\RR^{d}\to\RR$ is of class $C^2$, with bounded first and second order derivatives, there exist $\underline{a}$, $\overline{a}\in]0,\infty[$ such that $\underline{a}\leq a(x)\leq\overline{a}$ for all $x\in\RR^d$, and $b\colon\RR^d\to\RR^{d}$ is bounded,  of class $C^{2}$, with bounded first and second order derivatives. 
\end{remark}
\subsection{The first order mean field games system} Given $\mu\in C([0,T];\P_1(\RR^d))$, consider the HJB equation
\begin{align}
\label{e:HJB_u}
-\partial_{t}v(t,x) +H(x,D_{x}v(t,x)) &=F(x,\mu(t))\quad\text{for }(t,x)\in ]0,T[\times\RR^d,\nonumber\\
v(x,T)&= G(x,\mu(T))\quad\text{for }x\in\RR^d.
\end{align}
It follows from~\cite{BardiCapuzzo96,MR2784834} that~\eqref{e:HJB_u} admits a unique viscosity solution $v[\mu]$ and, for every $t\in[0,T[$, $x\in\RR^d$, and $\alpha\in L^{2}\big([t,T];\RR^{d}\big)$, setting $X^{t,x,\alpha}(\cdot)=x-\int_{t}^{(\cdot)}\alpha(s)\dd s$  and 
\begin{equation}
J^{t,x}[\mu](\alpha)=\int_{t}^{T}\Big(L\big(X^{t,x,\alpha}(s),\alpha(s)\big)+F\big(X^{t,x,\alpha}(s),\mu(s)\big)\Big)\dd s+G(X^{t,x,\alpha}(T),\mu(T)),
\end{equation}
we have
\begin{equation}
v[\mu](t,x)=\inf\Big\{J^{t,x}[\mu](\alpha)\,\big|\,\alpha\in L^{2}\big([t,T];\RR^{d}\big)\Big\}.
\end{equation}

The proof of the following result follows from standard arguments (see e.g.~\cite{CannSinesbook}). However, for the sake of completeness, we provide its proof in the appendix of this work. 
\begin{proposition}
\label{prop:value_function}
Assume {\bf(H\ref{h:h1})-(H\ref{h:h2})} and let $\mu\in C([0,T];\P_{1}(\RR^d))$. Then the following hold:
\begin{enumerate}[{\rm(i)}]
\item
\label{prop:value_function_i}
{\rm[Existence of an optimal control]} For every $(t,x)\in[0,T[\times\RR^d$, there exists $\alpha^{t,x}\in L^{\infty}\big([t,T];\RR^{d}\big)$ such that $v[\mu](t,x)=J^{t,x}[\mu](\alpha^{t,x})$. Moreover, there exists $C_{\text{{\rm b}}}>0$, independent of $(\mu,t,x)$, such that $\|\alpha^{t,x}\|_{L^{\infty}([0,T];\RR^d)}\leq C_{\text{{\rm b}}}$. 
\vspace{0.1cm}
\item
\label{prop:value_function_i_bis}
{\rm[Uniform bound]} We have 
\begin{equation}
|v[\mu](t,x)|\leq C_{\text{{\rm v}}}\quad\text{for all }(t,x)\in[0,T]\times\RR^{d},
\end{equation}
where $C_{\text{{\rm v}}}>0$ is independent of $\mu$.
\item 
\label{prop:value_function_ii}
{\rm[Lipschitz property]} We have 
\be 
\label{eq:value_function_Lipschitz}
\big|v[\mu](t,x)-v[\mu](t,y)\big|\leq C_{\text{{\rm Lip}}}|x-y|\quad\text{for all }t\in [0,T],\,x,\,y\in\RR^d,
\ee
where $C_{\text{{\rm Lip}}}>0$ is independent of $\mu$. 
\item 
\label{prop:value_function_iii}
{\rm[Semi-concavity]} We have 
\be 
\label{eq:value_function_sc}
v[\mu](t,x+y)-2v[\mu](t,x)+v[\mu](t,x-y)\leq C_{\text{{\rm sc}}}|y|^2\quad\text{for all }t\in[0,T],\,x,\,y\in\RR^d,
\ee
where $C_{\text{{\rm sc}}}>0$ is independent of $\mu$. 
\end{enumerate}
\end{proposition}

\begin{remark}
\label{rem:alternative_HJB_equation}
Assertion~\eqref{prop:value_function_i} in Proposition~\ref{prop:value_function} implies that, for every $\mu\in C([0,T];\P_{1}(\RR^{d}))$, we have
\be 
\label{eq:alternative_value_function}
v[\mu](t,x)=\inf\Big\{J^{t,x}[\mu](\alpha)\,\big|\,\alpha\in L^{\infty}([0,T];\RR^{d}),\;
\|\alpha\|_{L^{\infty}([0,T];\RR^{d})}\leq C_{\text{{\rm b}}}\Big\}
\ee
for all $(t,x)\in [0,T[\times\RR^{d}$. In particular, $v[\mu]$ is also characterized by the HJB equation 
\begin{align}
\label{e:HJB_u_b}
-\partial_{t}v(t,x) +H_{\text{{\rm b}}}(x,D_{x}v(t,x)) &=F(x,\mu(t))\quad\text{for }(t,x)\in ]0,T[\times\RR^d,\nonumber\\
v(x,T)&= G(x,\mu(T))\quad\text{for }x\in\RR^d,
\end{align}
where 
\begin{equation}
\label{def:H_b}
H_{\text{{\rm b}}}(x,p)=\sup_{a\in\ov{B}(0,C_{\text{{\rm b}}})}\big\{\langle a,p\rangle-L(x,a)\big\}\quad\text{for all }x,\,p\in\RR^d.
\end{equation}
\end{remark}
Consider the set-valued map $D_{x}^{+}v[\mu]\colon[0,T]\times\RR^d\to 2^{\RR^{d}}$ defined by 
\begin{equation*}
\label{def:d_plus_v}
D_{x}^{+}v[\mu](t,x)=\Bigg\{p\in\RR^{d}\,\Big|\,\limsup_{y\to x}\frac{v[\mu](t,y)-v[\mu](t,x)-\langle p,y-x\rangle}{|y-x|}\leq 0\Bigg\}\quad\text{for all }(t,x)\in [0,T]\times\RR^{d}.
\end{equation*}
It follows from Proposition~\ref{prop:value_function}\eqref{prop:value_function_iii} and  \cite[Proposition~3.1.5 and Proposition 3.3.4]{CannSinesbook}  that $D_{x}^{+}v[\mu]$ takes nonempty and closed values and its graph is closed.  In particular, since Proposition~\ref{prop:value_function}\eqref{prop:value_function_iii} and \cite[Theorem~3.3.6]{CannSinesbook}  imply that $D_{x}^{+}v[\mu](t,x)\subset\ov{B}(0,C_{\text{{\rm Lip}}})$ for all $(t,x)\in [0,T]\times\RR^{d}$, by~\cite[Chapter 1, Corollary 1]{MR755330} we have that $D_{x}^{+}v[\mu]$ is upper-semicontinuous, i.e. for every $M\subset \RR^{d}$ closed, $D_{x}^{+}v[\mu]^{-1}(M)$ is closed. Therefore, $D_{x}^{+}v[\mu]$ is a Borel measurable set-valued map and hence admits a Borel measurable selection (see e.g.~\cite[Corollary~14.6]{MR1491362}). Notice that Proposition~\ref{prop:value_function}\eqref{prop:value_function_ii}, Rademacher's theorem, and \cite[Proposition~3.1.5]{CannSinesbook} imply that all the measurable selections of $D_{x}^{+}v[\mu]$ coincide almost everywhere in $[0,T]\times\RR^d$ and hence, hereafter, we will denote likewise by $D_{x}v[\mu]$  any choice among them. 

Let $\mathsf{p}\in]1,\infty[$ be as in~{\bf(H\ref{h:initial_condition})}. We say that $m\in L^{\mathsf{p}}([0,T]\times\RR^{d})$ solves the continuity equation
\begin{align}
\label{limit_continuity_equation_H}
\partial_t m-\diver\left(D_pH(x,D_{x}v[\mu])m\right)&=0\quad
\mbox{in }]0,T[\times\RR^d,\nonumber\\ 
m(0)&=m_{0}^{*}\quad\text{in }\RR^d,
\end{align}
if, for every $\varphi\in C_{0}^{\infty}(\RR^{d})$ and $t\in [0,T]$, we have 
\be
\label{solution_distributional_sense}
\int_{\RR^d}\varphi(x)m(t,x)\dd x=\int_{\RR^d}\varphi(x)m_{0}^{*}(x)\dd x-
\int_{0}^{t}\int_{\RR^d}\Big\langle D_{p}H(x,D_{x}v[\mu](s,x)),D\varphi(x)\Big\rangle m(s,x)\dd x\dd s. 
\ee
\begin{proposition} 
Assume~{\bf(H\ref{h:h1})}-{\bf(H\ref{h:initial_condition})} and let $\mu\in C([0,T];\P_{1}(\RR^{d}))$. Then~\eqref{solution_distributional_sense} admits a solution $m\in C([0,T];\P_1(\RR^d))\cap L^{\mathsf{p}}([0,T]\times\RR^{d})$ and there exists $\widetilde{C}>0$ such that
\begin{equation}
\label{e:bound_lp_continuity_equation}
\|m(t,\cdot)\|_{L^{\mathsf{p}}(\RR^{d})}\leq \widetilde{C}\|m_{0}^{*}\|_{L^{\mathsf{p}}(\RR^{d})}\quad\text{for all }t\in[0,T].
\end{equation}
If, in addition, for every $t\in[0,T]$, the functions $F(\cdot,\mu(t))$ and $G(\cdot,\mu(T))$ are differentiable, then the solution $m$ to~\eqref{limit_continuity_equation_H} is unique. 
\end{proposition}
\begin{proof}
The first assertion in the statement follows from Proposition~\ref{prop:stability_w_r_t_mu_n} below, while the second one follows by arguing as in the proof of~\cite[Lemma~1.10]{MR4214774}. The crucial steps in the latter are the use of the superposition principle in~\cite{MR2096794} for solutions to \eqref{limit_continuity_equation_H} and the fact that, under the differentiability assumptions over $F(\cdot,\mu(t))$ and $G	(\cdot,\mu(T))$, the optimal control problem $\inf\Big\{J^{0,x}[\mu](\alpha)\,\big|\,\alpha\in L^{2}\big([0,T];\RR^{d}\big)\Big\}$ admits a unique solution for almost every $x\in\RR^{d}$. 
\end{proof}

Finally, we say that $(v^{*},m^{*})$, with $m^*\in C([0,T];\P_1(\RR^d))\cap L^{\mathsf{p}}([0,T]\times\RR^{d})$, solves~\eqref{MFG} if $v^*=v[m^*]$ and $m^*$ solves~\eqref{limit_continuity_equation_H} with $\mu=m^*$.

\begin{proposition}
\label{prop:MFG_existence_solution}
Assume~{\bf(H\ref{h:h1})}-{\bf(H\ref{h:initial_condition})}. Then system~\eqref{MFG} admits a solution $(v^{*},m^{*})$. 
\end{proposition}
\begin{proof} A proof of this result, under slightly different assumptions, can be found, for instance, in~\cite[Section~1.3.4]{MR4214774}. In the present context, the result follows from Theorem~\ref{th:convergence_discrete_scheme} below. 
\end{proof}

A uniqueness result for solutions to~\eqref{MFG} can be shown under additional assumptions on the coupling terms $F$ and $G$. A sufficient condition is the so-called Lasry-Lions monotonicity condition which states that, for $h=F,\,G$, it holds 
\be
\label{h:monotonicity_condition}
\int_{\RR^{d}}\big(h(x,m_1)-h(x,m_2)\big)\dd\big(m_1-m_2)(x)\geq 0\quad\text{for all }m_1,\,m_2\in\P_1(\RR^d). 
\ee
\begin{proposition}
\label{prop:uniqueness_mfg_monotone_couplings}
 Assume~{\bf(H\ref{h:h1})}-{\bf(H\ref{h:initial_condition})}, the monotonicity condition~\eqref{h:monotonicity_condition} and that, for all $\nu\in\P_{1}(\RR^{d})$, the functions $F(\cdot,\nu)$ and $G(\cdot,\nu)$ are differentiable. Then system~\eqref{MFG} admits a unique solution. 
\end{proposition}
\begin{proof}
The existence of a solution  to~\eqref{MFG} follows from Proposition~\ref{prop:MFG_existence_solution} while, under~\eqref{h:monotonicity_condition} and the differentiability assumptions on $F(\cdot,\nu)$ and $G(\cdot,\nu)$, the proof of the uniqueness of the solution follows by arguing as in the proof of~\cite[Theorem~1.8]{MR4214774}.
\end{proof}
\section{A semi-Lagrangian scheme for the HJB equation}
\label{sec:sl_hjb}
Let $\mu\in C([0,T];\P_{1}(\RR^{d}))$. In this section we recall a standard semi-Lagrangian scheme to approximate the viscosity solution $v[\mu]$ to~\eqref{e:HJB_u}. Most of the results for the semi-Lagrangian scheme that will be needed in the remainder of the article follow similarly to those in the monograph~\cite{falconeferretilibro} and the contributions~\cite{MR3148086,CS15,Chowdhury_et_al_2022}. The principal differences come from our assumptions on $L$ in~{\bf(H\ref{h:h1})}, which allow us to consider cost functionals not covered in these references (see e.g. the example in the last paragraph of Remark~\ref{rem:H_and_L}). Therefore, we confine ourselves to explain the main changes in the proofs of the aforementioned properties. On the other hand, the estimate in Proposition~\ref{prop:bound_hessian_semiconcavity_of_v} below seems to be new and will play a key role later in this article. 

In order to define the scheme, let $N\in\NN^*$ be the number of time steps, let $\Delta t=T/N$ be the time step, let $\I_{\Delta t}=\{0,\hdots,N\}$,  let $\I_{\Delta t}^{*}=\I_{\Delta t}\setminus\{N\}$, let
$t_{k}=k\Delta t$ for all $k\in\I_{\Delta t}$, and set $\G_{\Delta t}=\{t_{k}\,|\,k\in\I_{\Delta t}\}$.  
Given a space step $\Delta x>0$ and $i=(i_{1},\hdots,i_{d})\in\ZZ^{d}$, define $\beta_{i}^{1}\colon\RR^{d}\to\RR$ as 
\begin{equation}
\label{def:beta_i_1}
\beta_{i}^{1}(z)=\prod_{l=1}^{d}\widehat{\beta}\Big(\frac{z_{l}}{\Dx}-i_{l}\Big)\quad\text{for all }z=(z_{1},\hdots,z_{d})\in\RR^{d},
\end{equation}
where 
\begin{equation}
\label{def:beta_i_ref}
\widehat{\beta}(\xi)=\max\{0,1-|\xi|\}\quad\text{for all }\xi\in\RR.
\end{equation}
Notice that $\beta_{i}^{1}\geq 0$, $\sum_{i\in\ZZ^{d}}\beta_{i}^{1}(x)=1$ for all $x\in\RR^{d}$ and, setting $x_{i}=i\Delta x$, we have $\beta_{i}^{1}(x_j)=1$, if $i=j$, and $\beta_{i}^{1}(x_j)=0$, otherwise. Let $\G_{\Delta x}=\{i \Delta x\,|\,i\in\ZZ^{d}\}$ be the uniform grid and, given $\phi\colon\G_{\Delta x}\to\RR$, define its interpolate as
$$
I^{1}[\phi](x)=\sum_{i\in\ZZ^{d}}\beta_{i}^{1}(x)\phi_{i}\quad\text{for all }x\in\RR^{d},
$$
where, for notational simplicity, we have set $\phi_{i}=\phi(x_{i})$. For every $\varphi\colon\RR^{d}\to\RR$ denote by $\varphi|_{\G_{\Delta x}}$ its restriction to $\G_{\Delta x}$. If $\varphi$ is of class $C^{2}$ and has bounded second order derivatives, it follows from~\cite[Remark~3.4.2]{quarteronivalli94} that 
\be
\label{interpolation_of_phi}
\|\varphi(x)-I[\varphi|_{\G_{\Delta x}}](x)\|_{\infty}\leq C_{\varphi}(\Dx)^{2}, 
\ee
where $C_{\varphi}>0$ depends only on $\varphi$. 

We consider the following fully-discrete semi-Lagrangian scheme: find $\{v_{k}\colon\G_{\Delta x}\to\RR\,|\,k\in\I_{\Delta t}\}$ such that 
\begin{align}
\label{SL_HJB}
v_{k,i}&=\SS_{k,i}^{\text{{\rm fd}}}[\mu](v_{k+1})\quad\text{for all }k\in\I_{\Dt}^*,\,i\in\ZZ^{d},\nonumber\\ 
v_{N,i}&=G(x_i,\mu(T))\quad\text{for all }i\in\ZZ^{d},
\end{align}
where, for every $\phi\colon\G_{\Delta x}\to\RR$, bounded, $k\in\I_{\Delta t}$, and $i\in\ZZ^{d}$,  
\be
\label{eq:SL_scheme}
\SS_{k,i}^{\text{{\rm fd}}}[\mu](\phi)=\inf_{a\in\ov{B}(0,C_{\text{{\rm b}}})}\left[\Delta t L(x_i,a)+I^{1}[\phi](x_{i}-\Delta t a) \right]+\Delta t F(x_i,\mu(t_k)).
\ee
Notice that, being explicit, scheme~\eqref{SL_HJB} admits a unique solution. By definition, $\SS^{\text{{\rm fd}}}[\mu]$ is {\it monotone}, i.e. for every $\phi^{1}$, $\phi^2\colon\G_{\Delta x}\to\RR$, bounded, with $\phi^{1}_{i}\leq\phi^{2}_{i}$ for all $i\in\ZZ^{d}$, we have that 
\be 
\label{eq:SL_monotone}
\SS_{k,i}^{\text{{\rm fd}}}[\mu](\phi^{1})\leq\SS_{k,i}^{\text{{\rm fd}}}[\mu](\phi^{2})
\quad\text{for all }k\in\I_{\Delta t}^{*},\,i\in\ZZ^{d}.
\ee
Moreover, using {\bf(H\ref{h:h1})} and {\bf(H\ref{h:h2})}, standard arguments (see e.g.~\cite[Section 5.2.3]{falconeferretilibro}) yield the following {\it consistency} property for $\SS^{\text{{\rm fd}}}[\mu]$: let $(\mu_{n})_{n\in\NN}\subset C([0,T];\P_{1}(\RR^{d}))$, $\mu\in C([0,T];\P_{1}(\RR^{d}))$, $\big((\Delta t_n,\Delta x_n)\big)_{n\in\NN}\subset ]0,\infty[^{2}$, $\big((t_{k_{n}},x_{i_{n}})\big)_{n\in\NN}\subset \G_{\Delta t_n}\times\G_{\Delta x_n}$, and $(t,x)\in ]0,T[\times\RR^{d}$ such that, as $n\to\infty$, $\mu_{n}\to\mu$, $(\Delta t_{n},\Delta x_{n})\to 0$, $(\Delta x_{n})^{2}/\Delta t_{n}\to 0$, and $(t_{k_{n}},x_{i_{n}})\to (t,x)$. Then, recalling the definition of $H_{\text{{\rm b}}}$ in~\eqref{def:H_b}, for every $\varphi\colon[0,T]\times\RR^{d}\to\RR$ of class $C^{1}$, with bounded derivatives, we have  
\be
\label{eq:consistency_property}
\lim_{n\to \infty}\frac{1}{\Delta t_n}\Big(\varphi(x_{i_n},t_{k_n})-\SS_{k_{n},i_{n}}^{\text{{\rm fd}}}[\mu_{n}]\big(\varphi(t_{k_{n}+1},\cdot)|_{\G_{\Delta x_n}}\big)\Big)=-\partial_{t}\varphi(x,t)+H_{\text{{\rm b}}}(x,D_{x}\varphi(t,x))-F(x,\mu(t)).
\ee

Given $(\Delta t,\Delta x)\in ]0,\infty[^{2}$, let us set 
\begin{equation}
v^{\Delta t,\Delta x}[\mu](t_{k},x)=I^{1}[v_{k}](x)\quad\text{for all }k\in\I_{\Delta t},\,x\in\RR^{d},
\end{equation}
where, for every $k\in\I_{\Delta t}$, $v_{k}\colon\G_{\Delta x}\to\RR$ is computed with~\eqref{eq:SL_scheme}. We extend this definition to $[0,T]\times\RR^{d}$, by setting
\begin{equation}
\label{def:v_delta_t_delta_x}
v^{\Delta t,\Delta x}[\mu](t,x)=v^{\Delta t,\Delta x}[\mu](t_{k},x)\quad\text{if }t\in[t_{k},t_{k+1}[,\,k\in\I_{\Delta t}^{*}. 
\end{equation}
The following result provides properties for $v^{\Delta t,\Delta x}[\mu]$ that are analogous to those in Proposition~\ref{prop:value_function}\eqref{prop:value_function_i_bis}-\eqref{prop:value_function_iii} for $v[\mu]$.
\begin{proposition}
\label{prop:value_function_discrete}
Assume {\bf(H\ref{h:h1})-(H\ref{h:h2})}, let $\mu\in C([0,T];\P_{1}(\RR^d))$, and let $(\Delta t,\Delta x)\in ]0,\infty[^{2}$. Then the following hold:
\begin{enumerate}[{\rm(i)}]
\item
\label{prop:value_function_discrete_i}
{\rm[Stability]} We have 
\begin{equation}
|v^{\Delta x,\Delta t}[\mu](t,x)|\leq \widetilde{C}_{\text{{\rm v}}}\quad\text{for all }(t,x)\in [0,T]\times\RR^{d}, 
\end{equation}
where $\widetilde{C}_{\text{{\rm v}}}>0$ is independent of $(\mu,\Delta t,\Delta x)$.
\vspace{0.1cm} 
\item 
\label{prop:value_function_discrete_ii}
{\rm[Lipschitz property]} We have 
\be 
\label{eq:value_function_discrete_Lipschitz}
\big|v^{\Delta t,\Delta x}[\mu](t,x)-v^{\Delta t,\Delta x}[\mu](t,y)\big|\leq \widetilde{C}_{\text{{\rm Lip}}}|x-y|\quad\text{for all }t\in[0,T],\,x,\,y\in\RR^d,
\ee
where $\widetilde{C}_{\text{{\rm Lip}}}>0$ is independent of
$(\mu,\Delta t,\Delta x)$. 
\vspace{0.1cm} 
\item
\label{prop:value_function_discrete_iii}
{\rm[Discrete semi-concavity]} We have 
\begin{multline}
\label{eq:value_function_discrete_sc}
v^{\Delta t,\Delta x}[\mu](t,x+x_{i})-2v ^{\Delta t,\Delta x}[\mu](t,x)+v^{\Delta t,\Delta x}[\mu](t,x-x_{i})\\
\leq\widetilde{C}_{\text{{\rm sc}}}|x_{i}|^2\quad\text{for all }t\in[0,T],\,x\in\RR^{d},\,i\in\ZZ^d,
\end{multline}
where $\widetilde{C}_{\text{{\rm sc}}}>0$ is independent of
$(\mu,\Delta t,\Delta x)$. 
\end{enumerate}
\end{proposition}
\begin{proof}
\eqref{prop:value_function_discrete_i}: This follows directly from ~\eqref{SL_HJB},~\eqref{h:L_bounded_above_quadratic_term},
~\eqref{h:L_bounded_below_quadratic_term},~\eqref{h:F_bounded},
~\eqref{h:G_bounded}, and iteration.  

\eqref{prop:value_function_discrete_ii}: It follows from~\eqref{h:L_Lipschitz} that 
\begin{equation}
\label{eq:L_Lipschitz_restricted_control}
|L(x,a)-L(y,a)|\leq C_{L,3}(1+C_{\text{{\rm b}}}^{2})|x-y|\quad\text{for all }x,\,y\in\RR^{d},\,a\in\ov{B}(0,C_{\text{{\rm b}}}).
\end{equation}
Using this inequality,~\eqref{h:F_Lipschitz}, and~\eqref{h:G_Lipschitz}, the result follows from the same arguments than those in~\cite[Lemma~3.1{\rm(i)}]{CS15} (see also the proof of~\cite[Lemma~5.3(a)]{Chowdhury_et_al_2022}). 

\eqref{prop:value_function_discrete_iii}: It follows from~\eqref{h:L_second_order_derivative_x_bounded_above} that 
\begin{equation}
L(x+y,a)-2L(x,a)+L(x-y,a)\leq C_{L,5}(1+C_{\text{{\rm b}}}^{2})|y|^2\quad\text{for all }x,\,y\in\RR^{d},\,a\in\ov{B}(0,C_{\text{{\rm b}}}).
\end{equation}
In turn, using~\eqref{h:F_semiconcave} and~\eqref{h:G_semiconcave}, the result follows by arguing as in the proof of~\cite[Lemma~4.1]{MR3180719} (see also the proof of~\cite[Lemma~3.2(ii)]{Chowdhury_et_al_2022}). 
\end{proof}

Using the monotonicity of $\SS^{\text{{\rm fd}}}[\mu]$, the consistency property in ~\eqref{eq:consistency_property}, and the stability result in Proposition~\ref{prop:value_function_discrete}\eqref{prop:value_function_discrete_i}, the Barles-Souganidis relaxed limit method (see~\cite{MR1115933}) yields the following convergence result (see~\cite[Theorem~3.3]{MR3148086} for a detailed proof). 
\begin{proposition} 
\label{prop:convergence_SL_scheme}
Assume {\bf(H\ref{h:h1})-(H\ref{h:h2})}, let $(\mu_{n})_{n\in\NN}\subset C([0,T];\P_{1}(\RR^{d}))$, and let $\big((\Delta t_n,\Delta x_n)\big)_{n\in\NN}\subset ]0,\infty[^{2}$. Suppose that, as $n\to\infty$, $\mu_{n}\to\mu$, for some $\mu\in C([0,T];\P_{1}(\RR^{d}))$, $(\Delta t_{n},\Delta x_{n})\to 0$, and $(\Delta x_{n})^{2}/\Delta t_{n}\to 0$. Then $(v^{\Delta t_n,\Delta x_n}[\mu_{n}])_{n\in\NN}$ converges to $v[\mu]$ uniformly over compact subsets of $[0,T]\times\RR^{d}$. 
\end{proposition}

Given $\eps>0$, consider the mollifier $\RR^d\ni x\mapsto \rho_{\eps}(x)=\rho(x/\eps)/\eps^d\in\RR^d$, where $\rho\in C^{\infty}(\RR^d)$ has bounded derivatives of any order and satisfies $\rho(\RR^d)\subset[0,\infty[$ and $\int_{\RR^d}\rho(x)\dd x=1$. Given $\varphi\in W^{1,\infty}(\RR^{d})$, a standard computation shows that 
\begin{align}
\sup_{x\in\RR^{d}}\big|(\rho_{\eps}*\varphi)(x)-\varphi(x)\big|&\leq\eps\|D\varphi\|_{L^{\infty}(\RR^{d})},
\label{eq:infty_norm_difference_eps}\\
\sup_{x\in\RR^{d}}\big\|D^{\ell}(\rho_{\eps}*\varphi)(x)\big\|&\leq c_{\ell}\eps^{1-\ell}\quad\text{for all }\ell\in\NN,
\label{eq:bounds_higher_order_derivatives_eps}
\end{align}
where $\big\|D^{\ell}(\rho_{\eps}*\varphi)(x)\big\|$ denotes the operator norm of $D^{\ell}(\rho_{\eps}*\varphi)(x)$ and $c_{\ell}>0$ depends only on $\ell$. Let us set $\Delta=(\Delta t,\Delta x,\eps)$ and define 
\begin{equation}
\label{eq:v_epsilon}
v^{\Delta}[\mu](t,\cdot)=\rho_{\eps}*v^{\Delta t,\Delta x}[\mu](t,\cdot)
\quad\text{for all }t\in[0,T].
\end{equation}
The function $v^{\Delta}[\mu]$ satisfies similar properties than $v^{\Delta t,\Delta x}[\mu]$, as the following proposition shows. 
\begin{proposition}
\label{prop:value_function_discrete_eps}
Assume {\bf(H\ref{h:h1})-(H\ref{h:h2})}, let $\mu\in C([0,T];\P_{1}(\RR^{d}))$, and let $\Delta=(\Delta t,\Delta x,\eps)\in]0,\infty[^{3}$. Then the following holds:
\label{prop:properties_v_eps}
\begin{enumerate}[{\rm(i)}]
\item
\label{prop:value_function_discrete_eps_i}
{\rm[Stability]} We have 
\begin{equation}
|v^{\Delta}[\mu](t,x)|\leq \widetilde{C}_{\text{{\rm v}}}\quad\text{for all }(t,x)\in [0,T]\times\RR^{d},
\end{equation}
with $\widetilde{C}_{\text{{\rm v}}}>0$ being as in Proposition~\ref{prop:value_function_discrete}\eqref{prop:value_function_discrete_i}.
\vspace{0.1cm} 
\item 
\label{prop:value_function_discrete_eps_ii}
{\rm[Lipschitz property]} We have 
\be 
\label{eq:value_function_discrete_Lipschitz_eps}
\big|v^{\Delta}[\mu](t,x)-v^{\Delta}[\mu](t,y)\big|\leq \widetilde{C}_{\text{{\rm Lip}}}|x-y|\quad\text{for all }t\in[0,T],\,x,\,y\in\RR^d,
\ee
with $\widetilde{C}_{\text{{\rm Lip}}}>0$ being as in Proposition~\ref{prop:value_function_discrete}\eqref{prop:value_function_discrete_ii}.
\vspace{0.1cm} 
\item 
\label{prop:value_function_discrete_eps_iii}
{\rm[Approximate semi-concavity]} We have 
\begin{multline}
\label{eq:value_function_discrete_semiconcave_eps}
v^{\Delta}[\mu](t,x+y)-2v^{\Delta}[\mu](t,x)+v^{\Delta}[\mu](t,x-y)\\
\leq\widetilde{C}_{\text{{\rm asc}}}\Bigg(|y|^2+(\Delta x)^2+
\frac{(\Delta x)^2}{\eps}\Bigg)\quad\text{for all }t\in[0,T],\,x,\,y\in\RR^{d},
\end{multline}
where $\widetilde{C}_{\text{{\rm asc}}}>0$ is independent of
$(\mu,\Delta)$. 
\end{enumerate}
\end{proposition}
\begin{proof} Assertions~\eqref{prop:value_function_discrete_eps_i} and \eqref{prop:value_function_discrete_eps_ii} follow directly from~\eqref{eq:v_epsilon} and the corresponding assertions in Proposition~\ref{prop:value_function_discrete}. The proof of~\eqref{prop:value_function_discrete_eps_iii} follows from  Proposition~\ref{prop:value_function_discrete}\eqref{prop:value_function_discrete_iii} and arguing as in the proof of~\cite[Lemma~4.2]{MR3180719} (see also the proof of~\cite[Lemma~5.5(b)]{Chowdhury_et_al_2022}).
\end{proof}

In the following, given $A\subset\RR^{d}$, we denote by $\mathbb{I}_{A}$ the indicator function of $A$. The convergence result in the following proposition will play an important role in the next section.

\begin{proposition}
\label{prop:convergence_v_eps} Assume {\bf(H\ref{h:h1})-(H\ref{h:h2})}, let 
$(\mu_{n})_{n\in\NN}\subset C([0,T];\P_{1}(\RR^{d}))$, and let $\big((\Delta t_n,\Delta x_n,\eps_{n})\big)_{n\in\NN}\subset ]0,\infty[^{3}$. Set $\Delta_{n}=(\Delta t_n,\Delta x_n,\eps_{n})$ and let $\mu\in C([0,T];\P_1(\RR^{d}))$. Suppose that, as $n\to\infty$, $\mu_{n}\to\mu$, $\Delta_{n}\to 0$, and $(\Delta x_{n})^{2}/\Delta t_{n}\to 0$. Then the following hold:
\begin{enumerate}[{\rm(i)}]
\item
\label{prop:convergence_v_eps_i}
$(v^{\Delta_n}[\mu_{n}])_{n\in\NN}$ converges to $v[\mu]$ uniformly over compact subsets of $[0,T]\times\RR^{d}$.
\item
\label{prop:convergence_v_eps_ii}
If, in addition, $\Delta x_{n}/\eps_{n}\to 0$, then, for every $K\subset [0,T]\times\RR^{d}$ compact and $q\in[1,\infty[$, 
\begin{equation}
\label{eq:pointwise_convergence_feedback}
\mathbb{I}_{K} D_{p}H(\cdot,D_{x}v^{\Delta_{n}}[\mu_{n}])\to \mathbb{I}_{K} D_{p}H(\cdot,D_{x}v[\mu])\quad\text{in }L^{q}([0,T]\times\RR^{d}). 
\end{equation}
\end{enumerate}
\end{proposition}
\begin{proof}
\eqref{prop:convergence_v_eps_i}: This follows from Proposition~\ref{prop:value_function_discrete_eps}\eqref{prop:value_function_discrete_eps_ii}, estimate~\eqref{eq:infty_norm_difference_eps}, and Proposition~\ref{prop:convergence_SL_scheme}.

\eqref{prop:convergence_v_eps_ii}: It follows from Proposition~\ref{prop:value_function_discrete_eps}\eqref{prop:value_function_discrete_eps_iii} and~\cite[Remark 6]{MR3180719} that 
\begin{equation}
\label{eq:osl_condition}
\Big\langle D_{x}v^{\Delta}[\mu](t,y)-D_{x}v^{\Delta}[\mu](t,x),y-x\Big\rangle\leq c\Bigg(|y-x|^{2}+\frac{(\Delta x)^{2}}{\eps^{2}}\Bigg)\quad\text{for all }t\in[0,T],\,x,\,y\in\RR^{d}.
\end{equation}
Using this inequality and arguing as in the proof of~\cite[Theorem~3.5]{MR3148086}, one deduces that, as $n\to\infty$, $D_{x}v^{\Delta_{n}}[\mu_{n}]\to D_{x}v[\mu]$ almost everywhere in $[0,T]\times\RR^{d}$. In turn, since $H$ is of class $C^{2}$, we get that $ D_{p}H(\cdot,D_{x}v^{\Delta_{n}}[\mu_{n}])\to D_{p}H(\cdot,D_{x}v[\mu])$ almost everywhere in $[0,T]\times\RR^{d}$. Thus,~\eqref{eq:pointwise_convergence_feedback} follows from Proposition~\ref{prop:value_function_discrete_eps}\eqref{prop:value_function_discrete_eps_ii} and Lebesgue's dominated convergence theorem.
\end{proof}

We conclude this section with a useful estimate for $D^{2}v^{\Delta}[\mu]$. 
\begin{proposition}
\label{prop:bound_hessian_semiconcavity_of_v}
Assume {\bf(H\ref{h:h1})-(H\ref{h:h2})}, let $\mu\in C([0,T];\P_{1}(\RR^{d}))$, and let $\Delta=(\Delta t,\Delta x,\eps)\in]0,\infty[^{2}\times]0,1[$. Then it holds that
\be
\label{eq:hessian_bound_sc_v} 
\big\langle D_{x}^{2}v^{\Delta}[\mu](t,x)y,y\big\rangle\leq \widetilde{C}_{\text{{\rm hb}}}\Bigg(1 + \frac{(\Delta x)^{2}}{\eps^4}\Bigg)|y|^{2}\quad\text{for all }t\in[0,T],\,x,\,y\in\RR^{d},
\ee 
where $\widetilde{C}_{\text{{\rm hb}}}>0$ is independent of $(\mu,\Delta)$. 
\end{proposition}
\begin{proof}
Let us fix $t\in[0,T]$ and $x$, $y\in\RR^{d}$. If $y=0$ the result is true and, hence, let us assume that $y\neq 0$ and set $\tau= \Delta x/(\sqrt{\eps}|y|)$. In what follows, $C>0$ denotes a constant, independent of $(\mu,\Delta)$ which may change from line to line. It follows from Proposition~\ref{prop:value_function_discrete_eps}\eqref{prop:value_function_discrete_eps_iii} that
\be
\label{eq:semiconcavidaddebilcondossdf} 
v^{\Delta}[\mu](t,x+\tau y)-2v^{\Delta}[\mu](t,x)+v^{\Delta }[\mu](t,x-\tau y)
\leq C\tau^{2}|y|^{2}.
\ee   
On the other hand, a Taylor expansion of order $4$ and \eqref{eq:bounds_higher_order_derivatives_eps} imply that
\begin{align}
v^{\Delta}[\mu](t,x+\tau y)&\geq v^{\Delta}[\mu](t,x)+\langle D_{x}v^{\Delta}[\mu](t,x),\tau y\rangle+\frac{1}{2}\langle D_{x}^{2}v^{\Delta}[\mu](t,x)\tau y, \tau y\rangle\nonumber\\
&\hspace{0.3cm}+\frac{1}{6}D_{x}^{3}v^{\Delta}[\mu](t,x)(\tau y)^{3}
-\frac{1}{\eps^{3}}C|\tau y|^{4}, \nonumber\\[4pt]
v^{\Delta}[\mu](t,x-\tau y)&\geq v^{\Delta}[\mu](t,x)-\langle D_{x}v^{\Delta}[\mu](t,x), \tau y\rangle+\frac{1}{2}\langle D_{x}^{2}v^{\Delta}[\mu](t,x)\tau y,\tau y \rangle\nonumber\\ 
&\hspace{0.3cm}-\frac{1}{6}D_{x}^{3}v^{\Delta}[\mu](t,x)(\tau y)^{3}
-\frac{1}{\eps^{3}}C|\tau y|^{4},\nonumber
\end{align}
Adding both inequalities, using~\eqref{eq:semiconcavidaddebilcondossdf} and the relation $\tau|y|=\Delta x/\sqrt{\eps}$, we get
\begin{equation}
\langle D_{x}^{2}v^{\Delta}[\mu](t,x)\tau y,\tau y\rangle\leq
C\Bigg(1+\frac{(\Delta x)^{2}}{\eps^{4}}\Bigg)\tau^{2}|y|^{2}.
\end{equation}
Dividing by $\tau^{2}$ yields~\eqref{eq:hessian_bound_sc_v}.
\end{proof}
\section{A Lagrange-Galerkin type scheme for the continuity equation} 
\label{sec:lg_continuity_equation} 
Given $\mu \in C([0,T];\P_1(\RR^d))$ and $\Delta=(\Delta t,\Delta x,\eps)\in ]0,\infty[^3$, let $v^{\Delta}[\mu]$ be defined as in \eqref{eq:v_epsilon}. Consider the continuity equation 
\begin{align}
\partial_t m - \diver\left(D_pH(x,D_{x}v^{\Delta}[\mu])m\right) &=0
\qquad\mbox{in }]0,T[\times\RR^d,\nonumber\\ 
m(0)&=m_0^*\qquad\text{in }\RR^d.\label{continuity_equation_H}
\end{align}
Since $H$ is of class $C^{2}$ and $D_{x}v^{\Delta}[\mu]$ is bounded and Lipschitz, 
by~\cite[Proposition 8.1.8]{Ambrosiogiglisav} equation \eqref{continuity_equation_H} admits a unique solution $m^{\Delta}[\mu] \in C\left([0,T];\P_{1}(\RR^d)\right)$, which can be represented as
\be
\label{representation_formula_continuity_equation}
m^{\Delta}[\mu](t)=\Phi^{\Delta}[\mu](0,t,\cdot)\sharp m_0^*\quad\text{for all }t\in [0,T],
\ee 
where, for all $s\in [0,T)$ and $x\in \RR^d$, $\Phi^{\Delta}[\mu](s,\cdot,x)$ denotes the unique solution to 
\begin{align}
\dot{X}(t)&=-D_pH(X(t),D_{x}v^{\Delta}[\mu](t,X(t)))
\qquad\text{for a.e. }t\in (s,T),\nonumber\\ 
X(s)&=x.\label{dynamique}
\end{align} 
Relations~\eqref{representation_formula_continuity_equation} and~\eqref{dynamique} imply that
\be
\label{representation_formula_continuity_equation_two_times}
m^{\Delta}[\mu](t)=\Phi^{\Delta}[\mu] (s,t,\cdot)\sharp 
m^{\Delta}[\mu](s)\quad\text{for all }s,\,t\in[0,T],\,s\leq t.
\ee

On the other hand, notice that~\eqref{eq:D_P_H_linear_growth} and the uniform bound $\|D_{x}v^{\Delta}[\mu](t,\cdot)\|_{L^{\infty}(\RR^{d})}\leq \widetilde{C}_{\text{{\rm Lip}}}$ for all $t\in [0,T]$, which follows from Proposition~\ref{prop:value_function_discrete_eps}\eqref{prop:value_function_discrete_eps_ii}, yield the existence of $C_{\text{{\rm bf}}}>0$, independent of $(\mu,\Delta)$, such that 
\begin{equation}
\label{eq:bounded_flow}
\sup_{(t,x)\in [0,T]\times\RR^{d}}\big|D_pH(x,D_{x}v^{\Delta}[\mu](t,x))\big|
\leq C_{\text{{\rm bf}}}.
\end{equation}
Relation~\eqref{representation_formula_continuity_equation_two_times} and estimate~\eqref{eq:bounded_flow} have two straightforward consequences. The first one is that, by~{\bf(H\ref{h:initial_condition})}{\rm(i)}, we have
\be
\label{compact_support_continuous_case} 
\supp \left(m^{\Delta}[\mu](t)\right)\subset \ov{B}(0,C^{*}+TC_{\text{{\rm bf}}}) \quad \text{for all $t\in[0,T]$}.
\ee
The second one, is the uniform equicontinuity of the familly 
$\big\{m^{\Delta}[\mu]\,\big|\,\Delta\in]0,\infty[^3\big\}$ in $C([0,T];\P_1(\RR^d))$. More precisely, using~\eqref{e:d_1_alternative}, an easy computation shows that 
\be
\label{equicontinuity_continuous_time}
d_1\left(m^{\Delta}[\mu](t), m^{\Delta}[\mu](s)\right)
\leq C_{\text{{\rm bf}}}|t-s|\quad\text{for all } s,\,t\in[0,T]. 
\ee

The purpose of the following two propositions is to provide some stability estimates for $m^{\Delta}[\mu]$, which, together with~\eqref{compact_support_continuous_case}  and~\eqref{equicontinuity_continuous_time},  motivate the forthcoming analysis for a LG discretization of~\eqref{continuity_equation_H}.
\begin{proposition}
\label{prop:stability_cas_continu_I}
Assume {\bf(H\ref{h:h1})}-{\bf(H\ref{h:initial_condition})}, let $\mu \in C([0,T];\P_1(\RR^d))$, and let $\Delta=(\Delta t,\Delta x,\eps)\in ]0,\infty[^3$. Then for every $c>0$ there exists  $\tilde{c}>0$, independent of $\mu$, such that, if $\Delta x\leq c\eps^{2}$, it holds that
\be
\label{uniform_bound_lp_continuous_case}
\|m^{\Delta}[\mu](t,\cdot)\|_{L^{\mathsf{p}}(\RR^{d})}\leq \tilde{c}\|m_0^*\|_{L^{\mathsf{p}}(\RR^{d})}\quad\text{for all }t\in[0,T].
\ee
\end{proposition}
\begin{proof}
Let $c>0$ and suppose that $\Delta x\leq c\eps^{2}$. Proposition~\ref{prop:bound_hessian_semiconcavity_of_v} implies that 
\be
\Big\langle D_{x}^{2}v^{\Delta}[\mu](t,x)y,y\Big\rangle\leq \widetilde{C}_{\text{{\rm hb}}}\Bigg(1 +c^{2}\Bigg)|y|^{2}
\quad\text{for all }t\in[0,T],\,x,\,y\in \RR^{d}.
\ee
Using this inequality, the proof of~\eqref{uniform_bound_lp_continuous_case} follows from exactly the same arguments than those in the proof of~\cite[Proposition 4.1]{FischerSilva}. 
\end{proof}
\begin{proposition}
\label{prop:stability_cas_continu_II} 
Assume {\bf(H\ref{h:h1})}-{\bf(H\ref{h:initial_condition})}, let 
$(\mu_{n})_{n\in \NN}\subset C([0,T];\P_1(\RR^d))$, and let
$\big((\Delta t_n, \Delta x_n,\eps_n)\big)_{n\in\NN}\subset (0,\infty)^3$. Set $\Delta_n=(\Delta t_n,\Delta x_n,\eps_n)$ and let $\mu\in C([0,T];\P_1(\RR^d))$. Suppose that, as $n\to\infty$, $\mu_n\to \mu$, $\Delta_n\to 0$, $(\Delta x_n)^2/\Delta t_{n}\to 0$, and $\Delta x_n=O(\eps_n^2)$. Then, up to some subsequence, the following hold:
\begin{enumerate}[{\rm(i)}]
\item
\label{prop:stability_cas_continu_II_i}
$(v^{\Delta_{n}}[\mu^n])_{n\in\NN}$ converges to $v[\mu]$, uniformly over compact subsets of $[0,T]\times\RR^{d}$, and, for every $K\subset [0,T]\times\RR^d$ compact and $q\in[1,\infty[$, $(\mathbb{I}_{K}D_{x}v^{\Delta_{n}}[\mu_n])_{n\in\NN}$ converges to $\mathbb{I}_{K}D_{x}v[\mu]$ in $L^{q}([0,T]\times\RR^d)$. 
\item
\label{prop:stability_cas_continu_II_ii}
 $\big(m^{\Delta_n}[\mu_n]\big)_{n\in \NN}$ converges in $C([0,T];\P_1(\RR^d))$ towards a solution to~\eqref{limit_continuity_equation_H}. Moreover, the convergence also hold weakly in $L^{\mathsf{p}}([0,T]\times\RR^d)$, if $\mathsf{p}<\infty$, and weakly$^*$ in $L^{\infty}([0,T]\times\RR^d)$, if $\mathsf{p}=\infty$.
\end{enumerate}
\end{proposition}
\begin{proof}  For every $n\in\NN$, let us set $v^{n}=v^{\Delta_n}[\mu_n]$ and $m^{n}=m^{\Delta_n}[\mu_n]$.

\eqref{prop:stability_cas_continu_II_i}: This follows from Proposition~\ref{prop:convergence_v_eps}.

\eqref{prop:stability_cas_continu_II_ii}: It follows from~\eqref{compact_support_continuous_case},~\cite[Proposition~7.1.5]{Ambrosiogiglisav}, \eqref{equicontinuity_continuous_time}, and the Ascoli-Arzel\`a theorem, that there exists $m^*\in C([0,T];\P_1(\RR^d))$ such that, as $n\to\infty$ and up to some subsequence, $m^{n}\to m^*$ in $C([0,T];\P_1(\RR^d))$. By Proposition~\ref{prop:stability_cas_continu_I}, the convergence also hold weakly in $L^{\mathsf{p}}([0,T]\times \RR^d)$, if $\mathsf{p}<\infty$, and weakly$^*$ in $L^{\infty}([0,T]\times \RR^d)$, if $\mathsf{p}=\infty$. 
Since $m^{n}$ solves~\eqref{continuity_equation_H}, for every $t\in [0,T]$, and $\varphi\in C^{\infty}_{0}(\RR^d)$, we have 
\be
\label{weak_form_continuity_v_eps}
\ba{rcl}
\ds\int_{\RR^d}\varphi(x)m^{n}(t,x)\dd x&=&
\ds\int_{\RR^d}\varphi(x)m_0^{*}(t,x)\dd x\\[10pt]
\;&\;&\ds+\int_{0}^{t}\int_{\RR^d}\Big\langle D\varphi(x),D_{p}H\big(x,D_{x}v^{n}(s,x)\big)\Big\rangle m^{n}(s,x)\dd x\dd s.
\ea
\ee
Thus, by~\eqref{prop:stability_cas_continu_II_i}, we can pass to the limit in the previous expression to  deduce that $m^*$ solves \eqref{limit_continuity_equation_H}. 
\end{proof}
\subsection{The Lagrange-Galerkin approximation} The main purpose of the this section is to provide some results in the vein of Propositions~\ref{prop:stability_cas_continu_I} and~\ref{prop:stability_cas_continu_II} for solutions $\mathsf{m}^{\Delta}[\mu]$ to a LG approximation of \eqref{continuity_equation_H} that we proceed to construct. 

Let $\mu\in C([0,T];\P_1(\RR^d))$ and let $\Delta=(\Delta t, \Delta x,\eps)\in ]0,\infty[^3$. For every  $k\in \I_{\Delta t}^{*}$ and $x\in\RR^d$, let $\mathsf{\Phi}_{k}^{\Delta}[\mu](x)$ be the explicit Euler approximation of $\Phi^{\Delta}[\mu](t_{k},t_{k+1},x)$, i.e. 
\be
\label{Phi_mu_eps_k}
\mathsf{\Phi}_{k}^{\Delta}[\mu](x) = x- \Delta t D_{p}H\big(x,D_{x}v^{\Delta}[\mu](t_k,x)\big)\quad\mbox{for all }x\in\RR^d.
\ee
As in~\cite{camsilva12}, we consider the following semi-discrete approximation of \eqref{representation_formula_continuity_equation_two_times}:
\begin{equation}
\ba{rcl}
m_{k+1}&=&\mathsf{\Phi}_{k}^{\Delta}[\mu]\sharp m_{k}\quad
\text{for all }k\in\I_{\Delta t}^{*},\\[4pt]
m_{0}&=&m_{0}^{*},
\ea
\end{equation}
or, equivalently, for every $k\in\I_{\Delta t}^{*}$ and every Borel function $\varphi\colon\RR^d\to \RR$, such that $\varphi\left(\mathsf{\Phi}_{k}^{\Delta}[\mu](\cdot)\right)$ is integrable with respect to $m_k$,   
\be\label{abstract_semi_discrete}
\int_{\RR^d}\varphi(x)\dd m_{k+1}(x)=\int_{\RR^d}\varphi\left(\mathsf{\Phi}_{k}^{\Delta}[\mu](x)\right)\dd m_{k}(x).
\ee
Following~\cite{morton}, which mainly deals with a LG approximation of the dual (or transport) equation associated to  \eqref{continuity_equation_H}, let us formaly deduce from \eqref{abstract_semi_discrete} a time-space approximation of \eqref{continuity_equation_H}. For every $i\in \ZZ^d$, set 
\be 
\label{d:e_i}
E_{i}=\{x\in \RR^d\,|\,|x-x_i|_{\infty}\leq \Delta x/2\}
\ee
and define
\be
\label{discretization_initial_condition}
m_{0,i}=\frac{1}{(\Delta x)^d}\int_{E_{i}}m_{0}^*(x)\dd x.
\ee
Given the regular mesh defined by $\{E_{i}\,|\,i\in\ZZ^d\}$, let $\{\beta_i\,|\,i\in\ZZ^d\}$ be a finite element basis. In the following, we look for an approximation $\mathsf{m}^{\Delta}[\mu]$ of the solution $m^{\Delta}[\mu]$ to \eqref{continuity_equation_H} such that  
\be
\label{d:lg_discrete_times}
\mathsf{m}^{\Delta}[\mu](t_k,x)=\sum_{j\in\ZZ^d}m_{k,j}\beta_j(x)\quad
\mbox{for all }k\in\I_{\Delta t},\,x\in \RR^d,
\ee
for some  constants  $\{m_{k,j}\,|\,k\in\I_{\Dt},\,j\in\ZZ^d\}$. In order to determine the latter, we  replace  $m_k$ and $m_{k+1}$  in \eqref{abstract_semi_discrete} by $\mathsf{m}^{\Delta}(t_k,\cdot)$ and $\mathsf{m}^{\Delta}(t_{k+1},\cdot)$, respectively, and, given $i\in\ZZ^{d}$, we take $\varphi=\beta_{i}$ to obtain the following equations  
\be
\label{weak_scheme_11}
\sum_{j\in\ZZ^d}m_{k+1,j}\int_{\RR^d}\beta_{j}(x)\beta_{i}(x)\dd x
=\sum_{j\in\ZZ^d}m_{k,j}\int_{\RR^d}\beta_{i}(\mathsf{\Phi}_{k}^{\Delta}[\mu](x))
\beta_{j}(x)\dd x\quad\text{for all }k\in\I_{\Dt}^*.
\ee 
In the context of second-order Fokker-Planck equations, scheme~\eqref{weak_scheme_11} has already been considered in~\cite{Calzola_Carlini_Silva_lg_second_order} to provide a high-order accurate LG scheme to solve second-order mean field games problems with smooth solutions. In this reference, the authors consider symmetric Lagrangian basis of odd order which preserve the mass but not the positivity of the initial condition $\{m_{0,i}\,|\,i\in\ZZ^d\}$. Since we aim to approximate solutions to~\eqref{continuity_equation_H}, which in general are not smooth, from now on we take $\beta_{i}=\beta_{i}^{0}:=\mathbb{I}_{E_{i}}$ for all $i\in\ZZ^d$. Under this choice,~\eqref{weak_scheme_11} and~\eqref{discretization_initial_condition} yield the following LG scheme for~\eqref{continuity_equation_H}:
\begin{align}
\ds m_{k+1,i}&=\frac{1}{(\Delta x)^d}\sum_{j\in\ZZ^{d}}m_{k,j}\int_{E_j}\beta_i^{0}(\mathsf{\Phi}_{k}^{\Delta}[\mu](x))\dd x\quad\text{for all }k\in\I_{\Delta t}^{*},\,i\in \ZZ^{d},
\label{lg_scheme}\\ 
m_{0,i}&=\frac{1}{(\Delta x)^d}\int_{E_{i}}m_{0}^{*}(x)\dd x\quad\text{for all }i\in\ZZ^{d}. 
\label{lg_scheme_initial_condition}
\end{align}
The scheme above is explicit and hence admits a unique solution. Interestingly, the scheme \eqref{lg_scheme}-\eqref{lg_scheme_initial_condition} coincides with the one proposed in~\cite{MR2771664} (see also~\cite{MR2826759}) to approximate solutions to continuity equations. Indeed, we have
\be
\int_{E_j}\beta_i^{0}(\mathsf{\Phi}_{k}^{\Delta}[\mu](x))\dd x=\int_{\RR^d}\mathbb{I}_{E_{j}\cap \mathsf{\Phi}_{k}^{\Delta}[\mu]^{-1}(E_{i})}(x)\dd x=\L^{d}\bigg(E_{j}\cap \mathsf{\Phi}_{k}^{\Delta}[\mu]^{-1}(E_{i})\bigg),
\ee
where $\L^{d}$ denotes the Lebesgue measure in $\RR^{d}$. Plugging this expression in~\eqref{lg_scheme} yields the scheme in~\cite[Section~2.2]{MR2771664}. Notice that our main results for solutions to~\eqref{lg_scheme}-\eqref{lg_scheme_initial_condition}, contained in Propositions~\ref{prop:stability} and~\ref{prop:stability_w_r_t_mu_n} below, do not follow from the results in~\cite{MR2771664,MR2826759}. Therefore, the analysis in this section provides a complementary study to the one in~\cite{MR2771664,MR2826759} for the approximation~\eqref{lg_scheme}-\eqref{lg_scheme_initial_condition} of continuity equations. 
\subsection{Properties of LG scheme}
\label{sec:prop_lg_scheme}
We begin with a preliminary result stating that the solution to~\eqref{lg_scheme}-\eqref{lg_scheme_initial_condition} is supported on a compact set, which is independent of the discretization parameters provided that $\Delta x$ is of the order of $\Delta t$.
\begin{lemma}
\label{l:uniform_compact_support_discrete_solution} 
Assume {\bf(H\ref{h:h1})}-{\bf(H\ref{h:initial_condition})}, let $\mu\in C([0,T];\P_1(\RR^d))$, let $\Delta=(\Delta t,\Delta x,\eps)\in ]0,\infty[^3$, and let 
$\{m_{k,i}\,|\,k\in\I_{\Delta t},\,i\in\ZZ^d\}$ be the solution to~\eqref{lg_scheme}-\eqref{lg_scheme_initial_condition}. Then for every $c>0$ there exists $\widetilde{C}^{*}>0$, independent of $\mu$, such that, if $\Delta x\leq c\Delta t$, for every $k\in\I_{\Delta t}$ we have $m_{k,i}=0$ if $x_i\notin\ov{B}_{\infty}(0,\widetilde{C}^{*})$.
\end{lemma}
\begin{proof} Let $c>0$ and suppose that $\Delta x\leq c\Delta t$. For every $k\in\I_{\Delta t}^{*}$, set $r_{k}=\sup\{|x_{i}|_{\infty}\,|\,m_{k,i}\neq 0,\,i\in\ZZ^d\}\in[0,\infty]$. By~\eqref{eq:bounded_flow},~\eqref{lg_scheme}, and~{\bf(H\ref{h:initial_condition})}{\rm(i)}, we have
\begin{equation}
r_{k+1}\leq r_{k}+\Delta t C_{\text{{\rm bf}}}+\frac{\Delta x}{2}\leq r_{k}+\Delta t\bigg(C_{\text{{\rm bf}}}+\frac{c}{2}\bigg)\leq C^{*}+\Delta x N_{\Delta t}\Delta t\bigg(C_{\text{{\rm bf}}}+\frac{c}{2}\bigg)= C^{*}+T\bigg(C_{\text{{\rm bf}}}+\frac{c}{2}\bigg),
\end{equation}
for all $k\in\I_{\Delta t}^{*}$. The result follows by letting $\widetilde{C}^{*}=C^{*}+T(C_{\text{{\rm bf}}}+c/2)$.
\end{proof}
Let $\mu\in C([0,T];\P_1(\RR^d))$ and let $\Delta=(\Delta t, \Delta x,\eps)\in ]0,\infty[^3$. As a consequence of the previous result, in~\eqref{lg_scheme} it suffices to compute $m_{i,k+1}$ for $i\in\I_{\Delta x}$, where 
\be 
\label{d:i_delta_x}
\I_{\Delta x}:=\big\{i \in \ZZ^d\,|\,x_i\in\ov{B}_{\infty}(0,\widetilde{C}^{*})\big\}.
\ee

Given the constants $\{m_{k,i}\,|\,k\in\I_{\Delta t},\, i\in\ZZ^d\}$, computed 
with~\eqref{lg_scheme}-\eqref{lg_scheme_initial_condition}, we extend 
$\mathsf{m}^{\Delta}[\mu]$, given by~\eqref{d:lg_discrete_times}, to $[0,T]\times\RR^d$ as follows:   
\begin{multline}
\label{m_delta_in_t}
\mathsf{m}^{\Delta}[\mu](t,x)=\left(\frac{t_{k+1}-t}{\Delta t}\right)
\mathsf{m}^{\Delta}[\mu](t_k,x)+\left(\frac{t-t_{k}}{\Delta t}\right)
\mathsf{m}^{\Delta}[\mu](t_{k+1},x)\\
\text{for all }k\in\I_{\Delta t}^{*},
\,t\in[t_{k},t_{k+1}),\,x\in\RR^d. 
\end{multline}
In the following proposition, we state, for later use, some simple properties of the solution to~\eqref{lg_scheme}-\eqref{lg_scheme_initial_condition}.
\begin{proposition} 
\label{prop:basic_properties_lg}
Assume {\bf(H\ref{h:h1})}-{\bf(H\ref{h:initial_condition})}, let $\mu \in C([0,T];\P_1(\RR^d))$, let $\Delta=(\Delta t,\Delta x,\eps)\in ]0,\infty[^3$, and let
$\{m_{k,i}\,|\,k\in\I_{\Delta t},\,i\in\ZZ^d\}$ be the solution to~\eqref{lg_scheme}-\eqref{lg_scheme_initial_condition}. Then the following hold:
\begin{enumerate}[{\rm(i)}]
\item 
\label{prop:basic_properties_lg_i}
$\mathsf{m}^{\Delta}[\mu](t,x)\geq 0$ for all $t\in[0,T]$, $x\in\RR^d$.
\item 
\label{prop:basic_properties_lg_ii}
Let $c>0$. If $\Delta x\leq c\Delta t$ and $\widetilde{C}^{*}>0$ is given by Lemma~\ref{l:uniform_compact_support_discrete_solution}, we have
\be 
\label{l:uniform_compact_support_extended_discrete_solution}
\supp\big(\mathsf{m}^{\Delta}[\mu](t,\cdot)\big)\subseteq\ov{B}_{\infty}(0,\widetilde{C}^{*})\quad\text{for all }t\in[0,T]. 
\ee
\item
\label{prop:basic_properties_lg_iii} 
Let $a=(a_{i})_{i\in \ZZ^d}\subset\RR$ and set $\varphi_{a}(x)=\sum_{i \in \ZZ^d}a_i\beta_{i}^{0}(x)$ for all $x\in \RR^d$. Then we have
\be
\label{weak_scheme_weak_form}
\int_{\RR^d}\varphi_{a}(x)\mathsf{m}^{\Delta}[\mu](t_{k+1},x)\dd x
=\int_{\RR^d}\varphi_{a}(\mathsf{\Phi}_{k}^{\Delta}[\mu](x))\mathsf{m}^{\Delta}[\mu](t_{k},x)\dd x\quad\text{for all }k\in\I_{\Delta t}^{*}.
\ee
\item
\label{prop:basic_properties_lg_iv}
$\int_{\RR^d}\mathsf{m}^{\Delta}[\mu](t,x)\dd x=1$ for all $t\in[0,T]$.
\end{enumerate}
\end{proposition}
\begin{proof}

\eqref{prop:basic_properties_lg_i}: Using that $m_0^*\geq 0$ and $\beta_i^{0}\geq 0$ for all $i\in\ZZ^d$, this assertion follows directly from~\eqref{lg_scheme}-\eqref{lg_scheme_initial_condition}.

\eqref{prop:basic_properties_lg_ii}: This follows from Lemma~\ref{l:uniform_compact_support_discrete_solution} and~\eqref{m_delta_in_t}.

\eqref{prop:basic_properties_lg_iii}: For every $k\in\I_{\Delta t}^{*}$, by~\eqref{lg_scheme}, we have
\begin{multline}
\int_{\RR^d}\varphi_{a}\mathsf{m}^{\Delta}[\mu](t_{k+1},x)\dd x
=\sum_{i\in\ZZ^{d}}\sum_{j\in\ZZ^d}a_{j}m_{k+1,i}\int_{\RR^d}\beta_{i}^{0}(x)\beta_{j}^{0}(x)\dd x
=\sum_{i\in\ZZ^{d}}a_{i}m_{k+1,i}(\Delta x)^{d}\\
=\sum_{i\in\ZZ^{d}}a_{i}\sum_{j\in\ZZ^{d}}m_{k,j}\int_{E_{j}}\beta_{i}^{0}(\mathsf{\Phi}_{k}^{\Delta}[\mu](x))\dd x 
= \sum_{j\in\ZZ^{d}}m_{k,j}\int_{\RR^d}\varphi_{a}(\mathsf{\Phi}_{k}^{\Delta}[\mu](x))\beta_{j}^{0}(x)\dd x\\
=\int_{\RR^d}\varphi_{a}(\mathsf{\Phi}_{k}^{\Delta}[\mu](x))\mathsf{m}^{\Delta}[\mu](t_{k},x)\dd x. 
\end{multline}
Notice that the changes of the order of summation above are justified by~\eqref{prop:basic_properties_lg_ii}.\smallskip\\
\eqref{prop:basic_properties_lg_iv}: By~\eqref{prop:basic_properties_lg_iii}, with $\varphi_a(x):=\sum_{i\in\ZZ^d}\beta_{i}^{0}(x)=1$, we obtain the result for $t=t_k$, with $k\in\I_{\Delta t}$. The result for every $t\in[0,T]$ follows from~\eqref{m_delta_in_t}.
\end{proof}
In what follows, given $\varphi\colon\RR^d\to\RR$ we set
\be 
\label{d:interpolation_0}
I^{0}[\varphi](x)=\sum_{i\in\ZZ^{d}}\varphi(x_i)\beta_{i}^{0}(x)\quad\text{for all }x\in\RR^d. 
\ee
We will need the following estimate in some of the proofs below.
\begin{lemma}
\label{l:difference_integral_interpolation} 
Assume {\bf(H\ref{h:h1})}-{\bf(H\ref{h:initial_condition})}, let $\mu \in C([0,T];\P_1(\RR^d))$, let $\Delta=(\Delta t,\Delta x,\eps)\in ]0,\infty[^3$, and, given $L>0$, let $\varphi\colon\RR^d\to\RR$ be $L$-Lipschitz. Then, for every $k\in\I_{\Delta t}^{*}$, we have
\be 
\label{e:estimate_int_phi_m_two_steps}
\Bigg|\int_{\RR^d}\varphi(x)\mathsf{m}^{\Delta}[\mu](t_{k+1},x)\dd x-
\int_{\RR^d}\varphi(\mathsf{\Phi}_{k}^{\Delta}[\mu](x))\mathsf{m}^{\Delta}[\mu](t_{k},x)\dd x\Bigg|\leq L\sqrt{d}\Delta x.
\ee
\end{lemma}
\begin{proof} Since $\sum_{i\in\ZZ^d}\beta_{i}^{0}(x)=1$ for all $x\in\RR^d$, we have 
\begin{multline}
\label{e:estimate_infty_norm_int_phi_0}
\big|\varphi(x)-I^{0}[\varphi](x)\big|=\bigg|\sum_{i\in\ZZ^d}(\varphi(x)-\varphi(x_i))\beta_{i}^{0}(x)\bigg|\\
\leq \sum_{i\in\ZZ^d}\big|\varphi(x)-\varphi(x_i)|\beta_{i}^{0}(x)
\leq \frac{L\sqrt{d}}{2}\Delta x\quad\text{for all }x\in\RR^d.
\end{multline}
It follows that $\|\varphi-I^{0}[\varphi]\|_{\infty}\leq (L\sqrt{d}/2)\Delta x$ and, hence, 
\be
\Bigg|\frac{1}{(\Delta x)^d}\int_{E_{i}}\varphi(x)\dd x-\varphi(x_i)\Bigg|
=\Bigg|\frac{1}{(\Delta x)^d}\int_{E_{i}}(\varphi(x)-\varphi(x_i))\dd x\Bigg| \leq \frac{L\sqrt{d}}{2}\Delta x,
\ee 
from which we deduce that, for every $k\in\I_{\Delta t}^{*}$ and $j\in\ZZ^d$,
\begin{multline}
\label{estimation_long_calcul_1}
\Bigg|\sum_{i\in\ZZ^d}\frac{1}{(\Delta x)^d}\int_{E_i}\varphi(x)\dd x\int_{E_j}\beta_{i}^{0}(\mathsf{\Phi}_{k}^{\Delta}[\mu](y))\dd y-\int_{E_j}I^{0}[\varphi]
(\mathsf{\Phi}_{k}^{\Delta}[\mu](x))\dd x\Bigg| \\
=\Bigg|\sum_{i\in\ZZ^d}\left(\frac{1}{(\Delta x)^d}\int_{E_i}\varphi(x)\dd x - \varphi(x_i)\right)\int_{E_j}\beta_{i}^{0}(\mathsf{\Phi}_{k}^{\Delta}[\mu](y))\dd y\Bigg|
\leq \frac{L\sqrt{d}}{2}\Delta x\sum_{i\in\ZZ^d}\int_{E_j}\beta_{i}^{0}(\mathsf{\Phi}_{k}^{\Delta}[\mu](y))\dd y\\
=\frac{L\sqrt{d}}{2}(\Delta x)^{d+1}.
\end{multline}
Therefore, by~\eqref{estimation_long_calcul_1} and~\eqref{e:estimate_infty_norm_int_phi_0}, we obtain
\be 
\label{e:estimate_long_calcul_lemma}
\Bigg|\sum_{i\in\ZZ^d}\frac{1}{(\Delta x)^d}\int_{E_i}\varphi(x)\dd x \int_{E_j}\beta_{i}^{0}(\mathsf{\Phi}_{k}^{\Delta}[\mu](y))\dd y-\int_{E_j}\varphi
(\mathsf{\Phi}_{k}^{\Delta}[\mu](x)) \dd x\Bigg|\leq L\sqrt{d}(\Delta x)^{d+1}.
\ee
Finally, from~\eqref{lg_scheme}, \eqref{e:estimate_long_calcul_lemma},
and Proposition~\ref{prop:basic_properties_lg}\eqref{prop:basic_properties_lg_iv}, we get 
\begin{multline}
\Bigg|\int_{\RR^d}\varphi(x)\mathsf{m}^{\Delta}[\mu](t_{k+1},x)\dd x-
\int_{\RR^d}\varphi(\mathsf{\Phi}_{k}^{\Delta}[\mu](x))\mathsf{m}^{\Delta}[\mu](t_{k},x)\dd x\Bigg|\\
=\Bigg|\sum_{i\in\ZZ^d}\int_{E_{i}}\varphi(x)\dd x\frac{1}{(\Delta x)^d}\sum_{j\in\ZZ^d}m_{k,j}\int_{E_j}\beta_{i}^{0}(\mathsf{\Phi}_{k}^{\Delta}[\mu](y))\dd y
-\sum_{i\in\ZZ^d}m_{k,i}\int_{E_{i}}\varphi(\mathsf{\Phi}_{k}^{\Delta}[\mu](x))\dd x\Bigg|\\
=\Bigg|\sum_{j\in\ZZ^d}m_{k,j}\Bigg(\sum_{i\in\ZZ^d}\frac{1}{(\Delta x)^d}\int_{E_{i}}\varphi(x)\int_{E_j}\beta_{i}^{0}(\mathsf{\Phi}_{k}^{\Delta}[\mu](y))\dd y
-\int_{E_{j}}\varphi(\mathsf{\Phi}_{k}^{\Delta}[\mu](x))\dd x\Bigg)\Bigg|\leq L\sqrt{d}(\Delta x),
\end{multline}
which shows~\eqref{e:estimate_int_phi_m_two_steps}. 
\end{proof}
In the next result, we study the equicontinuity of the family $\{\mathsf{m}^{\Delta}[\mu]\,|\,\Delta\in]0,\infty[^{3}\}$, under the condition that $\Delta x$ is, at most, of the order of $\Delta t$. 
\begin{proposition}
\label{prop:equicontinuity} 
Assume {\bf(H\ref{h:h1})}-{\bf(H\ref{h:initial_condition})}, let $\mu\in C([0,T];\P_1(\RR^d))$, and let $\Delta=(\Delta t,\Delta x,\eps)\in ]0,\infty[^{3}$. Then, for every $c>0$, if $\Delta x\leq c\Delta t$, we have
\be
\label{equicontinuity}
d_1(\mathsf{m}^{\Delta}[\mu](t),\mathsf{m}^{\Delta}[\mu](s))\leq (C_{\text{{\rm bf}}}+c\sqrt{d})|t-s|\quad
\text{for all }t,\,s\in[0,T],
\ee
where, for every $t\in[0,T]$, $\mathsf{m}^{\Delta}[\mu](t)\in\P_{1}(\RR^d)$ denotes the measure $\dd \mathsf{m}^{\Delta}[\mu](t)(x)=\mathsf{m}^{\Delta}[\mu](t,x)\dd x$.
\end{proposition}
\begin{proof}
Let $\varphi\in\text{Lip}_1(\RR^d)$ and define $\psi_{\varphi}\colon[0,T]\to\RR$ by
\be 
\psi_{\varphi}(t)=\int_{\RR^d}\varphi(x)\mathsf{m}^{\Delta}[\mu](t,x)\dd x\quad\text{for all }t\in[0,T].
\ee
It follows from~\eqref{m_delta_in_t} that $\psi_{\varphi}$ is continuous and affine on every interval $[t_{k},t_{k+1}]$ ($k\in\I_{\Delta t}^{*}$). Thus, $\psi_{\varphi}\in W^{1,\infty}(]0,T[)$ and 
\be 
\label{e:infty_norm_derivative_psi}
\Big\|\frac{\dd}{\dd t}\psi_{\varphi}\Big\|_{\infty}=\frac{1}{\Delta t}\max_{k\in\I_{\Delta}^{*}}\Bigg|\int_{\RR^d}\varphi(x)\bigg(\mathsf{m}^{\Delta}[\mu](t_{k+1},x)-\mathsf{m}^{\Delta}[\mu](t_k,x)\bigg)\dd x\Bigg|.
\ee
In order to estimate the right-hand side of~\eqref{e:infty_norm_derivative_psi}, fix $k\in\I_{\Dt}^*$ and notice that, by Lemma~\ref{l:difference_integral_interpolation}, ~\eqref{Phi_mu_eps_k}, ~\eqref{eq:bounded_flow}, Proposition~\ref{prop:basic_properties_lg}\eqref{prop:basic_properties_lg_iv}, and $\Delta x\leq c\Delta t$, we have
\begin{align}
\int_{\RR^d}\varphi(x)\bigg(\mathsf{m}^{\Delta}[\mu](t_{k+1},x)-\mathsf{m}^{\Delta}[\mu](t_k,x)\bigg)\dd x&\leq\int_{\RR^d}\bigg(\varphi(\mathsf{\Phi}_{k}^{\Delta}[\mu](x))-\varphi(x)\bigg)\mathsf{m}^{\Delta}[\mu](t_{k},x)\dd x+\sqrt{d}\Delta x\nonumber\\
&\leq\int_{\RR^d}\Big|\mathsf{\Phi}_{k}^{\Delta}[\mu](x)-x\Big|\mathsf{m}^{\Delta}[\mu](t_{k},x)\dd x+\sqrt{d}\Delta x\nonumber\\
&\leq (C_{\text{{\rm bf}}}+c\sqrt{d})\Delta t.
\label{equicon_in_tk}
\end{align}
Changing $\varphi$ by $-\varphi$ in the previous computation,~\eqref{equicon_in_tk} implies that 
\be
\Bigg|\int_{\RR^d}\varphi(x)\bigg(\mathsf{m}^{\Delta}[\mu](t_{k+1},x)-\mathsf{m}^{\Delta}[\mu](t_k,x)\bigg)\dd x\Bigg|\leq (C_{\text{{\rm bf}}}+c\sqrt{d})\Delta t.
\ee
and hence, by~\eqref{e:infty_norm_derivative_psi}, $\Big\|\frac{\dd}{\dd t}\psi_{\varphi}\Big\|_{\infty}\leq (C_{\text{{\rm bf}}}+c\sqrt{d})$. Thus, we deduce that
\be 
\label{equicon_in_tk_d1}
\int_{\RR^d}\varphi(x)\bigg(\mathsf{m}^{\Delta}[\mu](t,x)-\mathsf{m}^{\Delta}[\mu](s,x)\bigg)\dd x\leq (C_{\text{{\rm bf}}}+c\sqrt{d})|t-s|\quad \text{for all }t,\,s\in[0,T]
\ee
and~\eqref{equicontinuity} follows from~\eqref{e:d_1_alternative}.
\end{proof}

The following result state a stability property for $\mathsf{m}^{\Delta}[\mu]$ which is analogous to the one in Proposition~\ref{prop:stability_cas_continu_I} for $m^{\Delta}[\mu]$.
\begin{proposition}
\label{prop:stability}
Assume {\bf(H\ref{h:h1})}-{\bf(H\ref{h:initial_condition})}, let $\mu \in C([0,T];\P_1(\RR^d))$, and let $\Delta=(\Delta t,\Delta x,\eps)\in ]0,\infty[^3$. Then for every $c_{1}$, $c_{2}>0$, there exists $\widetilde{C}>0$, independent of $\mu$, such that, if $\Delta$ is small enough, $\Delta x\leq c_{1}\Delta t$, and $\Delta t\leq c_{2}\eps^{2}$, we have $\mathsf{m}^{\Delta}[\mu](t,\cdot)\in L^{\mathsf{p}}(\RR^d)$ for all $t\in[0,T]$ and
\be 
\label{stability}
\|\mathsf{m}^{\Delta}[\mu](t,\cdot) \|_{L^{\mathsf{p}}(\RR^{d})} \leq\widetilde{C}\|m_{0}^*\|_{L^{\mathsf{p}}(\RR^{d})}\quad\text{for all }t\in[0,T].
\ee 
\end{proposition} 
\begin{proof}
Let $c_{1}$, $c_{2}>0$, suppose that $\Delta x\leq c_{1}\Delta t$, $\Delta t\leq c_{2}\eps^{2}$, fix $k\in \I_{\Dt}^*$, and let $\widetilde{C}^{*}>0$ be as in Lemma~\ref{l:uniform_compact_support_discrete_solution}. Then, by~\eqref{eq:bounded_flow}, there exists $R>0$, independent of $(\mu,\Delta,k)$, such that
$\big(\mathsf{\Phi}_{k}^{\Delta}[\mu]\big)^{-1}\big(\ov{B}_{\infty}(0,\widetilde{C}^{*})\big)\subset B_{\infty}(0,R)$. The regularity of $H$ and
estimate~\eqref{eq:bounds_higher_order_derivatives_eps}, with $\ell=2$, yield the existence of $C_{R}>0$, independent of $(\mu,\Delta,k)$, such that, for every $x\in B_{\infty}(0,R)$, 
the norm of the matrix
\begin{multline}
\label{eq:calcul_derivee_du_champ_de_vecteurs}
D_x\big(D_{p}H(x,D_{x}v^{\Delta}[\mu](t_{k},x))\big)\\
=D^2_{p}H\Big(x,D_{x}v^{\Delta}[\mu](t_{k},x)\Big)D_{x}^2v^{\Delta}[\mu](t_{k},x)+ D_{xp}^{2}H\Big(x,D_{x}v^{\Delta}[\mu](t_{k},x)\Big),
\end{multline}
induced by the $2$-norm in $\RR^d$, is bounded by $C_{R}/\eps$. In particular,  $D_{p}H(\cdot,D_{x}v^{\Delta}[\mu](t_{k},\cdot))$ is $(C_{R}/\eps)$-Lipschitz on $B_{\infty}(0,R)$. Thus, expression~\eqref{Phi_mu_eps_k} and the inequality $\Delta t/\eps \leq c_{2}\eps$ imply that, if $\eps$ is small enough, there exists $C_{1}>0$, independent of $(\mu,\Delta,k)$, such that 
\be
\big|\mathsf{\Phi}_{k}^{\Delta}[\mu](x)-\mathsf{\Phi}_{k}^{\Delta}[\mu](y)\big|\geq C_{1} |x-y|\quad\text{for all } x,\,y\in B_{\infty}(0,R),
\ee
which implies that $\mathsf{\Phi}_{k}^{\Delta}[\mu]$ is injective on  $B_{\infty}(0,R)$, and, denoting by $I_d$ the $d\times d$ identity matrix,
\be
\label{eq:derivative_phi_k_eps}
D_{x}\mathsf{\Phi}_{k}^{\Delta}[\mu](x)= I_d-\Dt D_x[D_pH(x,D_{x}v^{\Delta}[\mu](t_k,x))]
\ee
is invertible for $x\in B_{\infty}(0,R)$. In particular, $\mathsf{\Phi}_{k}^{\Delta}[\mu]$ is a diffeomorphism  of $B_{\infty}(0,R)$ onto $\mathsf{\Phi}_{k}^{\Delta}[\mu](B_{\infty}(0,R))$. Let us suppose first that $\mathsf{p}\in]1,\infty[$.  By the change of variable formula, we have
\begin{multline}
\label{estimate_l2_norm_k+1_witD_pHhi}
\int_{\RR^{d}} \Big(\mathsf{m}^{\Delta}[\mu]\big(t_{k+1},\mathsf{\Phi}_{k}^{\Delta}[\mu](x)\big)\Big)^{\mathsf{p}}\dd x=\int_{B_{\infty}(0,R)} \Big(\mathsf{m}^{\Delta}[\mu]\big(t_{k+1},\mathsf{\Phi}_{k}^{\Delta}[\mu](x)\big)\Big)^{\mathsf{p}}\dd x\\
=\int_{\mathsf{\Phi}_{k}^{\Delta}[\mu](B_{\infty}(0,R))}\big(\mathsf{m}^{\Delta}[\mu](t_{k+1}, y)\big)^{\mathsf{p}}\Big|\mbox{det}\left(D_{x}\mathsf{\Phi}_{k}^{\Delta}[\mu]\Big(\mathsf{\Phi}_{k}^{\Delta}[\mu]^{-1}(y)\Big) \right)\Big|^{-1} \dd y.
\end{multline}
Using again that the norm of $D_x\big(D_{p}H(\cdot,D_{x}v^{\Delta}[\mu](t_{k},\cdot))\big)$ is bounded by $C_{R}/\eps$ on $B_{\infty}(0,R)$, relation~\eqref{eq:derivative_phi_k_eps} and a Taylor expansion for the determinant  yield the existence of $C_{2}>0$, independent of $(\mu,\Delta,k)$, such that 
\begin{equation*}
\Big|\mbox{det}\left(D_{x}\mathsf{\Phi}_{k}^{\Delta}[\mu](x)\right)-\Big(1-\Dt \mbox{Tr}\left(D_x\big(D_pH(x,D_{x}v^{\Delta}[\mu](t_k,x))\big)\right)\Big)\Big|\leq C_{2}(\Dt/\eps)^2\quad\text{for all }x\in B_{\infty}(0,R),
\end{equation*}
where, given $B\in \RR^{d\times d}$, $\text{Tr}(B)$ denotes its trace. In turn, we get the existence of $C_{3}>0$, independent of $(\mu,\Delta,k)$, such that 
\be
\label{e:expansion_inverse_of_determinant}
\Big|\,\big|\mbox{det}\left(D_{x}\mathsf{\Phi}_{k}^{\Delta}[\mu](x)\right)\big|^{-1}-\Big(1+\Dt \mbox{Tr}\left(D_x\big(D_pH(x,D_{x}v^{\Delta}[\mu](t_k,x))\big)\right)\Big)\Big|\leq C_{3}(\Dt/\eps)^2,
\ee
for all $x\in B_{\infty}(0,R)$. Since $\Delta x\leq c_{1}c_{2}\eps^2$, Proposition~\ref{prop:bound_hessian_semiconcavity_of_v} implies that $D_{x}^{2}v^{\Delta}[\mu](t_{k},x)-\widetilde{C}_{\text{{\rm hb}}}(1+(c_{1}c_{2})^{2})I_d$ is negative semidefinite. Using that $H(\cdot,\cdot)$ is of class $C^{2}$ and convex with respect to its second argument, it follows from~\eqref{eq:calcul_derivee_du_champ_de_vecteurs} and \cite[Lemma~1.6.4]{CannSinesbook} that there exists $C_{4}>0$, independent of $(\mu,\Delta,k)$, such that 
\be 
\label{e:trace_inequality}
\Tr(D_x[ D_{p}H(x,D_{x}v^{\Delta}[\mu](t_{k},x))])\leq C_{4}\quad\text{for all }x\in B_{\infty}(0,R),
\ee
which, together with~\eqref{e:expansion_inverse_of_determinant}, yields
\be
\label{e:expansion_inverse_of_determinant_with_bound}
\Big|\mbox{det}\left(D_{x}\mathsf{\Phi}_{k}^{\Delta}[\mu](x)\right)\Big|^{-1}\leq1+C_{5}\Delta t\quad\text{for all }x\in B_{\infty}(0,R), 
\ee
where $C_{5}=C_{4}+C_{3}c_{2}$. Therefore, by~\eqref{estimate_l2_norm_k+1_witD_pHhi}, we get
\be
\label{estimate_l2_norm_k+1_witD_pHhi_final_bound}
\int_{\RR^{d}}\Big(\mathsf{m}^{\Delta}[\mu]\big(t_{k+1},
\mathsf{\Phi}_{k}^{\Delta}[\mu](x)\big)\Big)^{\mathsf{p}}\dd x\leq(1+C_{5}\Delta t)\int_{\RR^d} \big(\mathsf{m}^{\Delta}[\mu](t_{k+1}, x)\big)^{\mathsf{p}}\dd x.
\ee
Setting $\mathsf{p}^{*}=\mathsf{p}/(\mathsf{p}-1)$, it follows from~\eqref{weak_scheme_weak_form} and H\"older's inequality that 
\begin{align}
\|\mathsf{m}^{\Delta}[\mu](t_{k+1},\cdot)\|_{L^{\mathsf{p}}(\RR^{d})}^{\mathsf{p}}&=\int_{\RR^d}
\big(\mathsf{m}^{\Delta}[\mu](t_{k+1},x)\big)^{\mathsf{p}-1}\mathsf{m}^{\Delta}[\mu](t_{k+1},x)
\dd x\nonumber\\ 
&=\int_{\RR^d}\big(\mathsf{m}^{\Delta}[\mu](t_{k+1},\mathsf{\Phi}_{k}^{\Delta}[\mu](x))\big)^{\mathsf{p}-1} \mathsf{m}^{\Delta}[\mu](t_{k},x)\dd x\nonumber\\ 
&\leq\ds\left(\int_{\RR^{d}}\big(\mathsf{m}^{\Delta}[\mu](t_{k+1},\mathsf{\Phi}_{k}^{\Delta}[\mu](x))\big)^{\mathsf{p}}\dd x\right)^{\frac{1}{\mathsf{p}^{*}}}\|\mathsf{m}^{\Delta}[\mu](t_{k},\cdot)\|_{L^{\mathsf{p}}(\RR^{d})}\nonumber\\
&\leq (1+C_{5}\Dt)^{\frac{1}{\mathsf{p}^{*}}}\|\mathsf{m}^{\Delta}[\mu](t_{k+1},\cdot)\|_{L^{\mathsf{p}}(\RR^{d})}^{\frac{\mathsf{p}}{\mathsf{p}^{*}}}\|\mathsf{m}^{\Delta}[\mu](t_{k},\cdot)\|_{L^{\mathsf{p}}(\RR^{d})}\nonumber\\ 
&\leq (1+C_{5}\Dt)\|\mathsf{m}^{\Delta}[\mu](t_{k+1},\cdot)\|_{L^{\mathsf{p}}(\RR^{d})}^{\mathsf{p}-1}\|\mathsf{m}^{\Delta}[\mu](t_{k},\cdot)\|_{L^{\mathsf{p}}(\RR^{d})}.
\label{e:estimate_norm_lp_power_p}
\end{align}
In turn, we deduce that
\be
\|\mathsf{m}^{\Delta}[\mu](t_{k+1},\cdot)\|_{L^{\mathsf{p}}(\RR^{d})}\leq(1+C_{5}\Dt)\|\mathsf{m}^{\Delta}[\mu](t_{k},\cdot)\|_{L^{\mathsf{p}}(\RR^{d})}.
\ee 
By~\eqref{lg_scheme_initial_condition} and Jensen's inequality, we have $\|\mathsf{m}^{\Delta}[\mu](0,\cdot)\|_{L^{\mathsf{p}}(\RR^{d})}\leq\|m_0^*\|_{L^{\mathsf{p}}(\RR^{d})}$, and hence
\begin{align} 
\label{test}
\|\mathsf{m}^{\Delta}[\mu](t_{k+1},\cdot)\|_{L^{\mathsf{p}}(\RR^{d})}&\leq
(1+C_{5}\Dt)^{N}\|\mathsf{m}^{\Delta}[\mu](0,\cdot)\|_{L^{\mathsf{p}}(\RR^{d})}\nonumber\\
&\leq e^{C_{5}T}\|m_0^*\|_{L^{\mathsf{p}}(\RR^{d})},
\end{align}
which, by~\eqref{m_delta_in_t}, shows~\eqref{stability}, with $\widetilde{C}= e^{C_{5}T}$. If $\mathsf{p}=\infty$, then~\eqref{stability} holds for every $\mathsf{p}'\in]1,\infty[$. Noticing that $\widetilde{C}$ is independent of $\mathsf{p}'$ and that, by Proposition~\ref{prop:basic_properties_lg}\eqref{prop:basic_properties_lg_ii}, for every $t\in[0,T]$, the support of $\mathsf{m}^{\Delta}[\mu](t,\cdot)$ is contained in $\ov{B}_{\infty}(0,\widetilde{C}^{*})$,~\eqref{stability} for $\mathsf{p}=\infty$ follows by letting $\mathsf{p}'\to\infty$. 
\end{proof} 

The next result provides the analogous for $\mathsf{m}^{\Delta}[\mu]$ of Proposition~\ref{prop:stability_cas_continu_II} for $m^{\Delta}[\mu]$.

\begin{proposition}
\label{prop:stability_w_r_t_mu_n}
Assume {\bf(H\ref{h:h1})}-{\bf(H\ref{h:initial_condition})}, let 
$(\mu_{n})_{n\in \NN}\subset C([0,T];\P_1(\RR^d))$, and let
$\big((\Delta t_n, \Delta x_n,\eps_n)\big)_{n\in\NN}\subset (0,\infty)^3$. Set $\Delta_n=(\Delta t_n,\Delta x_n,\eps_n)$ and let $\mu\in C([0,T];\P_1(\RR^d))$. Suppose that, as $n\to\infty$, $\mu_n\to \mu$, $\Delta_n\to 0$, $\Delta x_n=o(\Delta t_n)$, and $\Delta t_n=O(\eps_n^2)$. Then, up to some subsequence, the following hold:
\begin{enumerate}[{\rm(i)}]
\item
\label{prop:stability_w_r_t_mu_n_i}
$(v^{\Delta_{n}}[\mu^n])_{n\in\NN}$ converges to $v[\mu]$, uniformly over compact subsets of $[0,T]\times\RR^{d}$, and, for every $K\subset [0,T]\times\RR^d$ compact and $q\in[1,\infty[$, $(\mathbb{I}_{K}D_{x}v^{\Delta_{n}}[\mu_n])_{n\in\NN}$ converges to $\mathbb{I}_{K}D_{x}v[\mu]$ in $L^{q}([0,T]\times\RR^d)$. 
\item
\label{prop:stability_w_r_t_mu_n_ii}
$\big(\mathsf{m}^{\Delta_n}[\mu_n]\big)_{n\in \NN}$  converges in $C([0,T];\P_1(\RR^d))$ towards a solution $m\in C([0,T];\P_1(\RR^d))\cap L^{\mathsf{p}}([0,T]\times\RR^{d})$ to~\eqref{limit_continuity_equation_H}. Moreover, the convergence also hold weakly in $L^{\mathsf{p}}([0,T]\times\RR^d)$, if $\mathsf{p}<\infty$, and weakly$^*$ in $L^{\infty}([0,T]\times\RR^d)$, if $\mathsf{p}=\infty$. In addition, there exists $\widetilde{C}>0$ such that 
\be 
\label{eq:stability_limit_solution}
\|m(t,\cdot) \|_{L^{\mathsf{p}}(\RR^{d})} \leq\widetilde{C}\|m_{0}^*\|_{L^{\mathsf{p}}(\RR^{d})}\quad\text{for all }t\in[0,T].
\ee 
\end{enumerate}
\end{proposition}
\begin{proof}
Let us set $v^n:=v^{\Delta_n,\eps_n}[\mu^n]$, $\mathsf{m}^n=\mathsf{m}^{\Delta_n,\eps_n}[\mu_n]$, and $\mathsf{\Phi}^{n}_{k}=\mathsf{\Phi}^{\Delta_n,\eps_n}[\mu_n]$ for all $k\in\I_{\Delta t_n}^{*}$.

\eqref{prop:stability_w_r_t_mu_n_i}: This corresponds to Proposition~\ref{prop:stability_cas_continu_II}\eqref{prop:stability_cas_continu_II_i}.

\eqref{prop:stability_w_r_t_mu_n_ii}: It follows from Proposition~\ref{prop:basic_properties_lg}\eqref{prop:basic_properties_lg_ii},
~\cite[Proposition~7.1.5]{Ambrosiogiglisav}, Proposition~\ref{prop:equicontinuity},
and the Arzel\'a-Ascoli theorem, that there exists $m\in C([0,T]; \P_1(\RR^d))$ such that, as $n\to\infty$ and up to some subsequence, $\mathsf{m}^n\to m$ in $C([0,T]; \P_1(\RR^d))$. Moreover, by Proposition~\ref{prop:stability}, the convergence holds weakly, if $\mathsf{p}<\infty$, and weakly$^*$ in $L^{\infty}([0,T]\times \RR^d)$, if $\mathsf{p}=\infty$, and $m$ satisfies~\eqref{eq:stability_limit_solution}. It remains to show that $m$ solves~\eqref{limit_continuity_equation_H}. Let $t\in]0,T]$ and let $k(n)\in\I_{\Delta t_n}^*$  be such that $t\in]t_{k(n)}, t_{k(n)+1}]$. For every $\varphi\in C_{0}^{\infty}(\RR^d)$, we have 
\be
\label{e:somme_telescopique}
\int_{\RR^d}\varphi(x)\mathsf{m}^n\left(t_{k(n)},x\right)\dd x=\int_{\RR^d}\varphi(x)  \mathsf{m}^n(0,x)\dd x+\sum_{k=0}^{k(n)-1}\int_{\RR^d}\varphi(x)\big(\mathsf{m}^n(t_{k+1},x)-\mathsf{m}^n(t_{k},x)\big)\dd x.
\ee

Let $k\in\I_{\Delta t}^{*}$. Since~\eqref{Phi_mu_eps_k} and~\eqref{eq:bounded_flow} yield $|\mathsf{\Phi}_k^{n}(x)-x| =O(\Delta t_n)$ for all $x\in\supp(\varphi)$, by Lemma~\ref{l:difference_integral_interpolation}, a Taylor expansion, and Proposition~\ref{prop:basic_properties_lg}\eqref{prop:basic_properties_lg_iv}, we have
\begin{align}
\int_{\RR^d}\varphi(x)\big(\mathsf{m}^n(t_{k+1},x)-\mathsf{m}^n(t_{k},x)\big)\dd x&=\int_{\RR^d}\big(\varphi(\mathsf{\Phi}_k^{n}(x))-\varphi(x)\big)\mathsf{m}^n(t_k,x)\dd x+O(\Delta x_n)\nonumber\\
&=-\Dt_n\int_{\RR^d}\big\langle D\varphi(x),D_pH(x,D_{x}v^{n}(t_k,x))\big\rangle\mathsf{m}^n(t_k,x)\dd x
+ O(\Delta x_n)\nonumber\\
&\hspace{0.3cm}+O((\Dt_n)^2),
\end{align}
which, combined with \eqref{e:somme_telescopique}, yields
\begin{align}
\int_{\RR^d}\varphi(x)\mathsf{m}^n(t_{k(n)},x)\dd x&=
\int_{\RR^d}\varphi(x)\mathsf{m}^n(0,x)\dd x-
\Delta t_n\sum_{k=0}^{k(n)-1}\int_{\RR^d}\big\langle D\varphi(x),D_pH(x,D_{x}v^{n}(t_k,x))\big\rangle\mathsf{m}^n(t_k,x)\dd x\nonumber\\
&\hspace{0.3cm}+O\left(\frac{\Delta x_n}{\Delta t_n}+ \Delta t_n\right).
\label{teofinec1} 
\end{align}
Since $\varphi$ has a compact support, it follows from ~\eqref{eq:bounds_higher_order_derivatives_eps}, with $\ell=2$, that there exists $C_{\varphi}>0$ such that $\big\langle D\varphi(\cdot), D_pH(\cdot,D_{x}v^n(t_k,\cdot))\big\rangle$ is $(C_{\varphi}/\eps_n)$-Lipschitz. Thus, by Proposition~\ref{prop:equicontinuity}, for every $k\in\I_{\Delta t}^{*}$, we have
\begin{multline} 
\left|\int_{\RR^d}\big\langle D\varphi(x),D_pH(x,D_{x}v^n(t_k,x))\big\rangle\big(\mathsf{m}^n(s,x)-\mathsf{m}^n(t_{k},x)\big)\dd x\right|= O\Bigg(\frac{\Dt_n}{\eps_n}\Bigg)\quad\text{for all } s\in [t_{k},t_{k+1}].
\end{multline}
Recalling that $D_{x}v^n(s,x)=D_{x}v^n(t_k,x)$ for all $s\in[t_k,t_{k+1}[$ and $x\in\RR^d$, we obtain
\begin{multline}
\Delta t_n\int_{\RR^d}\big\langle D\varphi(x),D_pH(x,D_{x}v^{n}(t_k,x))\big\rangle\mathsf{m}^n(t_k,x)\dd x\\
= \int_{t_{k}}^{t_{k+1}}\int_{\RR^d}\big\langle D\varphi(x),D_pH(x,D_{x}v^n(s,x))\big\rangle\mathsf{m}^n(s,x)\dd x\dd s+O\left(\frac{(\Delta t_{n})^2}{\eps_{n}}\right)
\end{multline}
and hence, in view of \eqref{teofinec1}, we deduce that, for $n$ large enough,
\begin{multline}
\label{pre_limit}
\int_{\RR^d}\varphi(x)\mathsf{m}^n\big(t_{k(n)},x\big)\dd x\\
=\int_{\RR^d}\varphi(x)\mathsf{m}^n(0,x)\dd x-\int_{0}^{T}\int_{\RR^d}\mathbb{I}_{[0,t_{k(n)}]} \big\langle D\varphi(x),D_pH(x,D_{x}v^n(s,x))\big\rangle\mathsf{m}^n(s,x)\dd x\dd s
+O\bigg(\frac{\Delta x_n}{\Delta t_n}+ \frac{\Delta t_n}{\eps_n}\bigg).
\end{multline}
Finally, by~\eqref{prop:stability_w_r_t_mu_n_i}, 
\be
\mathbb{I}_{[0,t_{k(n)}]}(\cdot)\big\langle D\varphi(\cdot),D_pH(x,D_{x}v^n(\cdot,\cdot))\big\rangle\underset{n\to\infty}{\longrightarrow}\mathbb{I}_{[0,t]}(\cdot)\big\langle D\varphi(\cdot),D_pH(x,D_{x}v[m](\cdot,\cdot))\big\rangle, 
\ee
in $L^{q}([0,T]\times\RR^d)$, for every $q\in[1,\infty[$, and, hence, we can pass to the limit in~\eqref{pre_limit} to obtain that $m$ satisfies \eqref{limit_continuity_equation_H}.
\end{proof}
\section{A Lagrange-Galerkin scheme for the the mean field games system}
\label{sec:LG_scheme_MFG}
In this section, we combine the schemes discussed in Sections~\ref{sec:sl_hjb} and~\ref{sec:lg_continuity_equation} to obtain a scheme for system~\eqref{MFG} and we provide a convergence result. 

Let $\Delta=(\Delta t,\Delta x,\eps)\in ]0,\infty[^{3}$, let $\widetilde{C}^{*}>0$ be as in Lemma~\ref{l:uniform_compact_support_discrete_solution}, and define
\be 
\label{def:sets_S_Delta}
\mathfrak{D}^{\Delta t,\Delta x}=\Bigg\{\mu=(\mu_{k,i})\,\Big|\,\mu_{k,i}\geq 0,\,\sum_{j\in\ZZ^d}\mu_{k,j}(\Delta x)^d=1\, \text{for all }k\in\I_{\Delta t},\,i\in \I_{\Delta x}\Bigg\},
\ee
where $\I_{\Delta x}$ is defined in~\eqref{d:i_delta_x}. Notice that $\mathfrak{D}^{\Delta t,\Delta x}$ is a convex and compact subset of $\RR^{(N_{\Delta t}+1)\times (2N_{\Delta x}+1)^{d}}$. Given $\mu\in\mathfrak{D}^{\Delta t,\Delta x}$ define $\tilde{\mu}\in C([0,T];\P_1(\RR^d))$ as
\begin{multline} 
\label{def:extension_mu}
\dd\tilde{\mu}(t)(x)= \bigg(\frac{t-t_k}{\Delta t}\bigg)\sum_{i\in\I_{\Delta x}}\mu_{k+1,i}\beta_{i}^{0}(x)\dd x + \bigg(\frac{t_{k+1}-t}{\Delta t}\bigg)\sum_{i\in\I_{\Delta x}}\mu_{k,i}\beta_{i}^{0}(x)\dd x\\
\text{for all }k\in\I_{\Delta t}^{*},\,t\in[t_{k},t_{k+1}[.
\end{multline}
The discretization of~\eqref{MFG} that we propose is the following: find $\mu\in\mathfrak{D}^{\Delta t,\Delta x}$ such that
\be
\label{MFGfully}
\mu_{k,i}=\mathsf{m}^{\Delta}[\tilde{\mu}](t_k,x_i)\quad\mbox{for all }k\in\I_{\Delta t}^*, \, i\in\I_{\Delta x}, \tag{$\text{{\rm MFG}}^{\Delta}$}
\ee
where we recall that $\mathsf{m}^{\Delta}[\tilde{\mu}]$ is defined in \eqref{m_delta_in_t}.

\begin{theorem} 
\label{th:solution_discrete_mfg}
Assume that {\bf(H\ref{h:h1})}-{\bf(H\ref{h:initial_condition})} hold. Then, if $\Delta t/\eps$ is small enough, system~\eqref{MFGfully} admits at least one solution. 
\end{theorem}
\begin{proof}Consider the application $T\colon\mathfrak{D}^{\Delta t,\Delta x}\to \RR^{(N_{\Delta t}+1)\times (2N_{\Delta x}+1)^{d}}$ defined by
$$
(T(\mu))_{k,i}=\mathsf{m}^{\Delta}[\tilde{\mu}](t_k,x_i)\quad\mbox{for all }k\in\I_{\Delta t}^*, \, i\in\I_{\Delta x}.
$$
It follows from Proposition~\ref{prop:basic_properties_lg}\eqref{prop:basic_properties_lg_i},\eqref{prop:basic_properties_lg_ii},\eqref{prop:basic_properties_lg_iv} that $T(\mathfrak{D}^{\Delta t,\Delta x})\subseteq\mathfrak{D}^{\Delta t,\Delta x}$. Moreover, if $(\mu_{n})_{n\in\NN}\subset\mathfrak{D}^{\Delta t,\Delta x}$ converges to $\mu$ then the continuity of $L$, $F$, and $G$,  imply that, as $n\to\infty$,  $v_{k,i}^{\Delta t,\Delta x}[\mu_n]\to
v_{k,i}^{\Delta t,\Delta x}[\mu]$ for all $k\in\I_{\Delta t}$ and $i\in\ZZ^d$. Thus, $(v^{\Delta t,\Delta x}[\mu_n])_{n\in\NN}$, defined in~\eqref{def:v_delta_t_delta_x}, converges to $v^{\Delta t,\Delta x}[\mu]$ pointwisely and hence, by Lebesgue's dominated convergence, the sequence $(v^{\Delta}[\mu_n])_{n\in\NN}$, defined in~\eqref{eq:v_epsilon}, satisfies that $v^{\Delta}[\mu_n]\to v^{\Delta}[\mu]$ and $D_{x}v^{\Delta}[\mu_n]\to D_{x}v^{\Delta}[\mu]$ pointwisely. Consequently, given $k\in\I_{\Delta t}^{*}$, it follows from~\eqref{Phi_mu_eps_k} that $\mathsf{\Phi}_{k}^{\Delta}[\mu_n]\to\mathsf{\Phi}_{k}^{\Delta}[\mu]$ pointwisely.  In particular, $\beta_{i}^{0}(\mathsf{\Phi}_{k}^{\Delta}[\mu_n](x))\to \beta_{i}^{0}(\mathsf{\Phi}_{k}^{\Delta}[\mu](x))$ for all $x\in \RR^d\setminus\big(\mathsf{\Phi}_{k}^{\Delta}[\mu]^{-1}(\partial E_{i})\big)$. If $R>0$ is as in the proof of Proposition~\ref{prop:stability} and $\Delta t/\eps$ is small enough, we have that $\mathsf{\Phi}_{k}^{\Delta}[\mu]$ is a diffeomorphism of $B_{\infty}(0,R)$ onto $\mathsf{\Phi}_{k}^{\Delta}[\mu](B_{\infty}(0,R))$.  Therefore, since $\L^{d}(\partial E_i)=0$,  we have $\L^{d}\big(\mathsf{\Phi}_{k}^{\Delta}[\mu]^{-1}(\partial E_{i})\big)=0$ and hence $\beta_{i}^{0}(\mathsf{\Phi}_{k}^{\Delta}[\mu_n](x))\to \beta_{i}^{0}(\mathsf{\Phi}_{k}^{\Delta}[\mu](x))$ for almost every $x\in\RR^d$. Therefore, by Lebesgue's dominated convergence,
$$
\int_{E_{j}}\beta_{i}^{0}\big(\mathsf{\Phi}_{k}^{\Delta}[\mu_n](x)\big)\dd x
\underset{n\to\infty}{\longrightarrow} \int_{E_{j}}\beta_{i}^{0}\big(\mathsf{\Phi}_{k}^{\Delta}[\mu](x)\big)\dd x\quad\text{for all }k\in\I_{\Delta t}^{*}, \, i,\, j\in \ZZ^d.
$$
Altogether, it follows from~\eqref{lg_scheme} that, as $n\to\infty$, $T(\mu_n)\to T(\mu)$, i.e. $T$ is continuous. Finally, the existence of a solution to \eqref{MFGfully}, i.e. of a fixed point of $T$, follows from Brouwer's fixed-point theorem.
\end{proof}

In the next result we provide our main result, which shows the convergence, up to some subsequence, of solutions to~\eqref{MFGfully} towards a solution to~\eqref{MFG}.

\begin{theorem}
\label{th:convergence_discrete_scheme}
Assume {\bf(H\ref{h:h1})}-{\bf(H\ref{h:initial_condition})}, let
$\big((\Delta t_n, \Delta x_n,\eps_n)\big)_{n\in\NN}\subset ]0,\infty[^3$, and set $\Delta_n=(\Delta t_n,\Delta x_n,\eps_n)$. Suppose that, as $n\to\infty$, $\Delta_n\to 0$, $\Delta x_n=o(\Delta t_n)$, and $\Delta t_n=O(\eps_n^2)$. For every $n$, large enough, let 
$m^{n}\in \SS^{\Delta_n}$ be a solution to $(\text{{\rm MFG}}^{\Delta_{n}})$, define $\tilde{m}^{n}$  by~\eqref{def:extension_mu}, and set $v^{n}=v^{\Delta_n}[\tilde{m}^n]$. Then there exists a solution $(v^*,m^*)$ to~\eqref{MFG} such that, up to some subsequence, the following hold:
\begin{enumerate}[{\rm(i)}]
\item 
\label{th:convergence_discrete_scheme_i}
$(v^{n})_{n\in\NN}$ converges to $v^*$, uniformly over compact subsets of $[0,T]\times\RR^d$.
\item
\label{th:convergence_discrete_scheme_ii}
$(\tilde{m}^{n})_{n\in\NN}$ converges in $C([0,T];\P_1(\RR^d))$ towards $m^*$. Moreover, the convergence also hold weakly in $L^{\mathsf{p}}([0,T]\times\RR^d)$, if $\mathsf{p}<\infty$, and weakly$^*$ in $L^{\infty}([0,T]\times\RR^d)$, if $\mathsf{p}=\infty$. In addition, there exists $\widetilde{C}>0$ such that 
\be 
\label{eq:stability_limit_solution_mfg_case}
\|m^{*}(t,\cdot) \|_{L^{\mathsf{p}}(\RR^{d})} \leq\widetilde{C}\|m_{0}^*\|_{L^{\mathsf{p}}(\RR^{d})}\quad\text{for all }t\in[0,T].
\ee 
\end{enumerate}
\end{theorem}
\begin{proof} For all $n\in\NN$, large enough, we have $\tilde{m}^{n}=\mathsf{m}^{\Delta_n}[\tilde{m}^{n}]$. Thus, Proposition~\ref{prop:basic_properties_lg}\eqref{prop:basic_properties_lg_ii},~\cite[Proposition~7.1.5]{Ambrosiogiglisav}, Proposition~\ref{prop:equicontinuity}, and the Arzel\'a-Ascoli theorem, imply the existence of $m^*\in C([0,T]; \P_1(\RR^d))$ and  subsequence, still labelled by $n$, such that  $(\tilde{m}^n)_{n\in \NN}$ converges to $m^*$ in $C([0,T]; \P_1(\RR^d))$. It follows from Proposition~\ref{prop:stability_w_r_t_mu_n} that $m^*$ solves~\eqref{limit_continuity_equation_H}, with $\mu=m^*$, i.e. $(v[m^{*}],m^{*})$ solves~\eqref{MFG}. Therefore, assertions~\eqref{th:convergence_discrete_scheme_i}-\eqref{th:convergence_discrete_scheme_ii} follow from the corresponding assertions in Proposition~\ref{prop:stability_w_r_t_mu_n}.
\end{proof}
\begin{remark}
\label{rem:main_result}
\begin{enumerate}[{\rm(i)}]
\item
\label{rem:main_result_i} 
Theorem~\ref{th:convergence_discrete_scheme} shows, in particular, that system~\eqref{MFG} admits at least one solution $(v^*,m^*)$. If the solution to~\eqref{MFG} is unique, then the entire sequence $(v^{n},m^{n})$ converges to 
$(v^*,m^*)$ {\rm(}see Theorem~\ref{prop:uniqueness_mfg_monotone_couplings} for a sufficient condition ensuring uniqueness{\rm)}.
\item The condition on $\Delta x_{n}=o(\Delta t_{n})$ in Theorem~\ref{th:convergence_discrete_scheme} is stronger than the condition $(\Delta x_{n})^{2}=o(\Delta t_{n})$ needed for convergence, when the space dimension is equal to one, in the scheme studied in~\cite{MR3148086} {\rm(}see also~\cite{Chowdhury_et_al_2022}{\rm)}. This can be explained by the estimate~\eqref{e:estimate_int_phi_m_two_steps} in Lemma~\ref{l:difference_integral_interpolation}, which seems difficult to improve, even if $\varphi$ is smooth, and it is in compliance with Assumption~3.1 in~\cite{MR2826759}, which plays an important role in the LG approximation of continuity equations with Lipschitz vector fields. 
\label{rem:main_result_ii} 
\end{enumerate}
\end{remark}

\section{Numerical results}
\label{sec:num_results}
In this section, given $\Delta=(\Delta t,\Delta x,\eps)$, we use~\eqref{MFGfully} to approximate the solutions to two first order MFGs systems. In order to obtain an implementable version, we need to approximate the integrals in the LG scheme \eqref{lg_scheme}-\eqref{lg_scheme_initial_condition}. We consider two methods. In the first  one, the integrals are approximated by numerical quadrature, while, in the second one, we use the so-called {\it area weighting} technique, introduced in~\cite{morton} and recalled in Section~\ref{sec:area_weighting} below.   

In the first example, the state dimension is equal to one and the data of the MFGs system does not satisfy some of the assumptions in Section~\ref{sec:assumptions}. On the other hand, the PDE system admits an explicit solution, which allows to compare the quadrature and area weighting methods to solve~\eqref{MFGfully}. For comparable accuracies, the area weighting method is less expensive than the quadrature method and, hence, we use the former in order to treat the second example, where the state dimension is equal to two and no explicit solution is known. Let us point out that the data of the second example fulfills all the assumptions in Section~\ref{sec:assumptions}. 

We solve ~\eqref{MFGfully} heuristically by fixed point iterations that  are stopped as soon as the uniform norm of the difference between two consecutive iterates is smaller than a given threshold $\tau$, which in the simulations is set to $10^{-3}$. In particular, we use the classical Picard iterations in the first test, as in~\cite{MR3148086}, and Picard iterations with damping parameter 0.5 in the second test, as in  ~\cite{Chowdhury_et_al_2022}.

\subsection{Area-weighted LG approximation}
\label{sec:area_weighting}
Let $\mu\in C([0,T];\P_{1}(\RR^{d}))$ and consider the continuity equation~\eqref{continuity_equation_H}. The main idea of the area-weighting technique is to replace, for each $k\in\I_{\Delta x}^{*}$, the local nonlinear discrete flow $E_{i}\ni x\mapsto \mathsf{\Phi}_{k}^{\Delta}[\mu](x)\in\RR^{d}$, defined by~\eqref{Phi_mu_eps_k}, by the local affine approximation 
\be 
\label{eq:def_area_weighting}
E_{i}\ni x\mapsto \ov{\mathsf{\Phi}}_{k}^{\Delta}[\mu](x)=x-\Delta t D_{p}H\big(x_{i},D_{x}v^{\Delta}[\mu](t_k,x_{i})\big)\in\RR^{d}.
\ee
Notice that $\ov{\mathsf{\Phi}}_{k}^{\Delta}[\mu](x)=x-x_{i}+\mathsf{\Phi}_{k}^\Delta[\mu](x_{i})$ for all $x\in E_{i}$. Under this approximation, we can compute the integrals in~\eqref{lg_scheme}-\eqref{lg_scheme_initial_condition} explicitly. Indeed, for all $i=(i_{1},\hdots,i_{d})\in\ZZ^d$ and $l=1,\hdots,d$,  let us set
$\mathrm{T}_{i_{l}}=[(x_{i})_{l}-\Delta x/2,(x_{i})_{l}+\Delta x/2]$, and observe that 
\be
\label{def_beta}
\beta^{0}_{i}(y)=\prod^{d}_{l=1}\mathbb{I}_{\mathrm{T}_{i_{l}}}(y_{l})\quad\text{for all }y=(y_{1},\hdots,y_{d})\in\RR^{d}. 
\ee
It  follows from~\eqref{eq:def_area_weighting} and~\eqref{def_beta} that, for every $i,\,j\in\ZZ^{d}$, we have
\begin{align}
&\int_{E_{j}}\beta_{i}^{0}(\ov{\mathsf{\Phi}}_{k}^{\Delta}[\mu](y))\dd y
=\int_{E_j}\beta_{i}^{0}(y-x_j+\mathsf{\Phi}^{\Delta}_{k}[\mu](x_{j}))\dd y\nonumber\\ &=\prod_{l=1}^{d} \int_{(x_{j})_{l}-\Delta x/2}^{(x_{j})_{l}+\Delta x/2}\mathbb{I}_{\mathrm{T}_{i_{l}}} \left(y_{l}-(x_{j})_{l}+\big(\mathsf{\Phi}^{\Delta}_{k}[\mu](x_{j})\big)_{l}\right)\dd y_{l}
=\prod_{l=1}^{d}\int_{\big(\mathsf{\Phi}^{\Delta}_{k}[\mu](x_{j})\big)_{l}-\Delta x/2}^{\big(\mathsf{\Phi}^{\Delta}_{k}[\mu](x_{j})\big)_{l}+\Delta x/2}\mathbb{I}_{\mathrm{T}_{i_{l}}}(y_{l})\dd y_{l}\label{eq:produit_tensoriel_integral}\\
&=\prod_{l=1}^{d} \L^1\bigg(\left[(x_{i})_{l}-\Delta x/2,(x_{i})_{l}+\Delta x/2\right]\cap \left[(\mathsf{\Phi}^{\Delta}_{k}[\mu](x_{j}))_{l}-\Delta x/2,(\mathsf{\Phi}^{\Delta}_{k}[\mu](x_{j}))_{l}+\Delta x/2\right]\bigg).
\nonumber
\end{align}
On the other hand, for every $l=1,\hdots,d$, it follows from~\eqref{def:beta_i_ref} that 
\begin{align*}
&\L^1\bigg(\left[(x_{i})_{l}-\Delta x/2,x_{i})_{l}+\Delta x/2\right]\cap \left[(\mathsf{\Phi}^{\Delta}_{k}[\mu](x_{j}))_{l}-\Delta x/2,(\mathsf{\Phi}^{\Delta}_{k}[\mu](x_{j}))_{l}+\Delta x/2\right]\bigg)\\
&= \begin{cases} 
\Delta x+\big(\mathsf{\Phi}^{\Delta}_{k}[\mu](x_{j})\big)_{l}- (x_{i})_{l}& \text{if } 
\big(\mathsf{\Phi}^{\Delta}_{k}[\mu](x_{j})\big)_{l}\in [(x_{i})_{l}-\Delta x,(x_{i})_{l}],\\ 
 \Delta x
+(x_{i})_{l}-\big(\mathsf{\Phi}^{\Delta}_{k}[\mu](x_{j})\big)_{l} & \text{if }
 \big(\mathsf{\Phi}^{\Delta}_{k}[\mu](x_{j})\big)_{l}\in](x_{i})_{l},(x_{i})_{l}+\Delta x],\\
  0 & \text{otherwise},
 \end{cases}
\\
&=\Delta x \widehat\beta\left(\big(\mathsf{\Phi}^{\Delta}_{k}[\mu](x_{j})\big)_{l}/{\Delta x}-i_{l}\right),
\end{align*}
which, combined with~\eqref{def:beta_i_1} and~\eqref{eq:produit_tensoriel_integral}, yields 
\be
\label{eq:int}
\frac{1}{(\Dx)^d}\int_{E_{j}}\beta_{i}^{0}(\ov{\mathsf{\Phi}}_{k}^{\Delta}[\mu](y))\dd y=\beta^1_i(\mathsf\Phi^\Delta_k[\mu](x_j)).
\ee
Thus,  replacing $\mathsf{\Phi}_{k}^{\Delta}[\mu]$ by $\ov{\mathsf{\Phi}}_{k}^{\Delta}[\mu]$ in~\eqref{lg_scheme} and, for every $i\in\ZZ^{d}$, denoting by $m_{0,i}^{*}$ any approximation of $\int_{E_{i}}m_{0}^{*}(x)\dd x/(\Delta x)^{d}$, we obtain the following area-weigthed LG version of ~\eqref{lg_scheme}-\eqref{lg_scheme_initial_condition}:
\begin{align}
\ov{m}_{k+1,i}&=\sum_{j\in\ZZ^{d}}\ov{m}_{k,j}\beta^{1}_{i}(\mathsf\Phi^\Delta_{k}[\mu](x_{j}))\quad\text{for all }k\in\I_{\Delta t}^{*},\,i\in\ZZ^d,\nonumber\\
\ov m_{0,i}&=m^{*}_{0,i}\quad\text{for all }i\in \ZZ^d.
\label{sl_scheme}
\end{align}
\begin{remark}
\label{rem:area_weighting}
Notice that~\eqref{sl_scheme} corresponds to the scheme proposed in~\cite{MR3148086} for the continuity equation~\eqref{continuity_equation_H}. Therefore, the latter can be seen as an area-weighted version of the LG scheme of Section~\ref{sec:lg_continuity_equation}.
\end{remark}
\subsection{Non-local MFG with analytical solution.}
\label{sec:example_explicit_solution}
We consider system~\eqref{MFG} with a quadratic Hamiltonian $H(x,p)=\frac{p^2}{2}$, coupling terms
\begin{equation}
\label{eq:couplings_lq_example}
F(x,\nu)=\frac{1}{2}\Big(x-\int_{\RR^d}y\,\dd \nu(y)\Big)^2,\quad G(x,\nu)=0,
\end{equation}
and initial  data $m_0^*$ given by the distribution of a $d$-dimensional Gaussian random variable with mean $\mu^*\in\RR^d$ and covariance matrix $\Sigma_0\in\RR^{d\times d}$ assumed, for simplicity, to be diagonal. Notice that the coupling term $F$ in~\eqref{eq:couplings_lq_example} and the initial distribution $m_0^*$ do not satisfy assumptions {\bf(H2)} and {\bf(H3)}, respectively. On the other hand, the MFG system admits in this case an explicit solution, which allows to compare the performance of quadrature and area-weighting methods to approximate the continuity equation. Indeed, setting 
$$
\Pi(t)=\ds\left(\frac{e^{2T-t}-e^t}{e^{2T-t}+e^t}\right)I_d,\quad
s(t)=-\Pi(t)\mu^*,\quad 
c(t)=\frac{1}{2}\langle\Pi(t)\mu^*,\mu^*\rangle\quad\text{for all }t\in[0,T],
$$
and arguing as in~\cite[Section~5.2]{Calzola_Carlini_Silva_lg_second_order},
one finds that~\eqref{MFG} admits a unique solution $(v^{*},m^{*})$ given by 
$$
v^{*}(t,x)=\ds\frac{1}{2}\langle\Pi(t)x,x\rangle+\langle s(t),x\rangle +c(t)\quad\text{for all }(t,x)\in[0,T]\times\RR^{d}
$$
and, for every $t\in[0,T]$, $m^{*}(t)$ is the joint distribution of $d$ independent Gaussian random variables $\{X_{l}(t)\,|\,l=1,\hdots,d\}$ with means $\mu_{\ell}(t)$ and variances $\sigma^{2}_{\ell}(t)$ ($l=1,\hdots,d$), given by
$$
\mu_{\ell}(t)=\mu^{*}_{l}\quad\text{and}\quad\sigma^{2}_{\ell}(t)=\Bigg(\frac{e^{2T-t}+e^t}{e^{2T}+1}\Bigg)^{2}(\Sigma_{0})_{l,l}.
$$

In the numerical test, we take $T=0.25$, $d=1$, $\mu^{*}=0.1$, and $\Sigma_{0}=0.105$. Since the exact solution $m^{*}(t)$ does not have a compact support, we approximate the system on the bounded domain $\OO= ]-2,2[$ and we impose Dirichlet boundary conditions at $x=-2$ and $x=2$, which are equal to the values of the exact solution at these points. In order to implement the latter, we proceed as in~\cite[Section 5.1.5]{falconeferretilibro}. 

We test our scheme for different values of $\Delta x$, the time step is chosen as $\Dt=(\Dx)^{2/3}/2$, and the mollifier in \eqref{eq:v_epsilon} is defined with $\RR\ni x\mapsto \rho(x)=e^{-x^{2}/2}/\sqrt{2\pi}\in\RR$ and $\eps=\sqrt{\Dt}$. We denote by $(v^{\Delta},\mathsf{m}^{\Delta})$ and $(\ov{v}^{\Delta},\ov{\mathsf{m}}^{\Delta})$ the approximations of solutions to~\eqref{MFGfully} obtained by estimating the integrals in~\eqref{lg_scheme} by numerical quadrature and by the area-weighting method, respectively. For the numerical quadrature of the integrals in the computation of $(v^{\Delta},\mathsf{m}^{\Delta})$, we divide each interval $E_{i}$ into $\lfloor 4/\Dx\rfloor$ subintervals and we use the midpoint rule on each one of them. The initial condition $m_{0}^{*}$ being smooth, we use the midpoint rule to approximate the integrals in~\eqref{lg_scheme_initial_condition}.

Setting $\G_{\Dx}(\ov{\OO}):=\G_{\Dx}\cap\ov{\OO}$, Tables~\ref{table_1} and~\ref{table_2} below show the uniform and $L^2$ relative discrete errors  
\be
\label{eq:errors}
\displaystyle
E_\infty(h^\Delta)=\frac{\underset{x_i\in\G_{\Dx}(\Omega)}\max|h^\Delta(x_i)-h(x_i)|}{\underset{x_i\in\G_{\Dx}(\Omega)}\max|h(x_i)|},\quad E_2(h_\Delta)=\left(\frac{\underset{x_i\in\G_{\Dx}(\Omega)}\sum |h^\Delta(x_i)-h(x_i)|^2}  {\underset{x_i\in\G_{\Dx}(\Omega)}\sum|h(x_i)|^2}\right)^\frac{1}{2},
\ee
for $(h,h^{\Delta})=(m^{*}(T,\cdot),\mathsf{m}^{\Delta}(T,\cdot))$, $(m^{*}(T,\cdot),\ov{\mathsf{m}}^{\Delta}(T,\cdot))$, $(v^{*}(0,\cdot),v^{\Delta}(0,\cdot))$,
and $(v^{*}(0,\cdot),\ov{v}^{\Delta}(0,\cdot))$. Our results show smaller errors for the approximations computed with numerical quadrature, specially in the uniform norm. However, such precision is achieved at the expense of a high computational cost compared with the area-weighted approximation. Table~\ref{table_2} shows that the higher precision obtained by computing an approximation of $m^*$ by numerical quadrature does not significantly affect the approximation of the value function $v^{*}$.

\begin{center}
\begin{table}[hbt!]
\captionsetup{skip=5pt} 
\begin{tabular}{ |c||c||c||c||c|}
 \hline 
 \rule{0pt}{12pt} & & & & \\[-14pt]
 $\Dx$ & $E_\infty(\mathsf{m}^{\Delta}(T,\cdot))$ & $E_\infty(\ov{\mathsf{m}}^{\Delta}(T,\cdot))$ &$E_{2}(\mathsf{m}^{\Delta}(T,\cdot))$ &$E_2(\ov{\mathsf{m}}^{\Delta}(T,\cdot))$ \\[1pt]
 \hline
4.80 $\cdot 10^{-2}$   & 8.41 $\cdot 10^{-3}$  & 3.69 $\cdot 10^{-2}$  &   6.30 $\cdot 10^{-3}$ & 1.09  $\cdot 10^{-2}$  \\

2.40 $\cdot 10^{-2}$  & 6.91 $\cdot 10^{-3}$   & 3.25 $\cdot 10^{-2}$   & 4.39 $\cdot 10^{-3}$ & 1.05 $\cdot 10^{-2}$ \\

1.20 $\cdot 10^{-2}$  &3.94 $\cdot 10^{-3}$  & 2.62$\cdot 10^{-2}$&  2.77 $\cdot 10^{-3}$  & 6.77 $\cdot 10^{-3}$ \\

6.00 $\cdot 10^{-3}$   &1.83 $\cdot 10^{-3}$  & 2.44 $\cdot 10^{-2}$&  6.89 $\cdot 10^{-4}$ & 2.67 $\cdot 10^{-3}$\\
 \hline
\end{tabular}
\caption{Errors for the approximation of $m^{*}(T,\cdot)$.}
\label{table_1}
\end{table}
\end{center}
\begin{center}
\begin{table}[hbt!]
\captionsetup{skip=5pt} 
\begin{tabular}{ |c||c||c||c||c|}
 \hline
\rule{0pt}{12pt} & & & & \\[-14pt]
  $\Dx$ & $E_\infty( v^{\Delta}(0,\cdot))$ & $E_\infty(\ov{v}^{\Delta}(0,\cdot))$ & $E_2(v^{\Delta}(0,\cdot))$ & $E_2( \ov{v}^{\Delta}(0,\cdot))$\\[1pt]
 \hline
4.80 $\cdot 10^{-2}$   & 7.02 $\cdot 10^{-3}$  & 7.11 $\cdot 10^{-3}$  &   6.20 $\cdot 10^{-3}$ & 6.31  $\cdot 10^{-3}$  \\

2.40 $\cdot 10^{-2}$  & 5.74 $\cdot 10^{-3}$   & 5.82 $\cdot 10^{-3}$   & 4.90 $\cdot 10^{-3}$ & 5.12 $\cdot 10^{-3}$\\

1.20 $\cdot 10^{-2}$  &4.34 $\cdot 10^{-3}$  & 4.37 $\cdot 10^{-3}$&  3.70 $\cdot 10^{-3}$  & 3.75 $\cdot 10^{-3}$ \\

6.00 $\cdot 10^{-3}$   &3.30 $\cdot 10^{-3}$  & 3.36 $\cdot 10^{-3}$&  2.95 $\cdot 10^{-3}$ & 3.01 $\cdot 10^{-3}$\\
 \hline
\end{tabular}
\caption{Errors for the approximation of $v^{*}(0,\cdot)$.}
\label{table_2}
\end{table}
\end{center}
\subsection{A two-dimensional example} In this test, we consider system~\eqref{MFG} with $d=2$, a quadratic Hamiltonian $H(x,p)=|p|^{2}/2$,  and coupling terms having the form 
\be
\label{eq:coupling_terms_example_two_d}
F(x,\nu)= \gamma\min\{|x-\bar{x}|^{2},R\}+ (r_{\sigma}*\nu)(x)\quad\text{and}\quad G(x,\nu)=0,
\ee
where $\gamma>0$, $\bar{x}\in\RR^{2}$, $R>0$, and, for $\sigma>0$, $r_{\sigma}(x)=e^{-|x|^2/2\sigma^{2}}/(2\pi\sigma^{2})$ for all $x\in\RR^{2}$. Given $\ell>0$, $x_{0}^{*}\in]0,\ell[^{2}$, and $\sigma_{0}>0$, we consider the initial density
\be 
\label{eq:m_0_second_example}
m_{0}^{*}(x)= \frac{\chi(x)}{\int_{[0,\ell]^{2}}\chi(y)\dd y}\quad\text{with}\quad
\chi(x)=e^{-|x-x_{0}^{*}|^2/2\sigma_{0}^{2}}\mathbb{I}_{[0,\ell]^{2}}(x)\quad\text{for all }x\in\RR^{d}.
\ee
Notice that the data above satisfy {\bf(H\ref{h:h1}), (H\ref{h:h2})}, and {\bf(H\ref{h:initial_condition})}, with $p=\infty$. In our tests below, we choose $T=1$, $\ell=2$, $x_{0}^*=(0.75,0.75)$, $\sigma_0=0.07$ in the initial condition, $\bar x=(1.75,1.75)$,  $R=5$, $\sigma=0.25$, and two values $\gamma=0.5$ and $\gamma=3$ in the running cost $F$. Since in this two-dimensional example the computational cost to solve~\eqref{MFGfully} is important, in view of the discussion in Section~\ref{sec:example_explicit_solution} we implement the area-weighting method of Section~\ref{sec:area_weighting} to approximate the integrals in~\eqref{lg_scheme}. The integrals in~\eqref{lg_scheme_initial_condition}, to approximate the initial condition $m_{0}^{*}$, are computed by using the midpoint rule.  
We set  $\Dx=0.025$, $\Dt=\Dx^{2/3}$, and the mollifier in \eqref{eq:v_epsilon} is defined with $\RR^{2}\ni x\mapsto \rho(x)=e^{-|x|^2/2}/2\pi\in\RR$ and $\eps =\sqrt{\Dt}/2$. Figure 1 shows the approximation $\overline{\mathsf m}^\Delta$ of the exact distribution $m^*$ in the $x_1$-$x_2$ plane obtained after solving~\eqref{MFGfully} for $\gamma=0.5$ and $\gamma=3$. On the left, we display the evolution of the initial distribution, concentrated around $x_{0}^{*}$,  by overlaying the distributions 
$\overline{\mathsf m}^\Delta(t_k,\cdot)$ for $k \in \I_{\Dt}$. On the right, we display only the final distribution $\ov {\mathsf m}^\Delta(T,\cdot)$. The simulation shows the effect of the positive constant $\gamma$, which weights the importance of reaching the target point $\bar{x}$. If $\gamma=0.5$, the aversion to crowed regions, modeled by the second term in the definition of $F$, has a more relevant impact on the distribution of the players than the term penalizing the distance to $\bar{x}$, while, if $\gamma=3$, the latter term has a preponderant role in the evolution of the distribution of the agents. 

\begin{figure}[hbt!]
\centering
\begin{subfigure}{0.45\textwidth}
    \includegraphics[width=\textwidth]{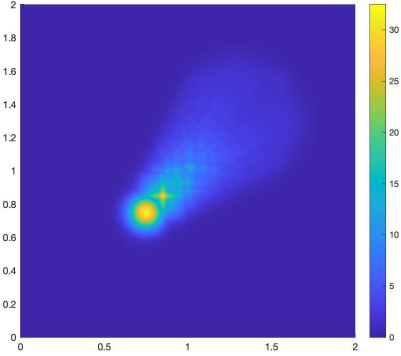}
    \caption{Time evolution $(\ov {\mathsf m}^\Delta(t_k,\cdot))_{k\in \I_{\Dt}}$  for $\gamma=0.5$.}
    \label{mass_K_1}
\end{subfigure}
\hfill
\begin{subfigure}{0.45\textwidth}
    \includegraphics[width=\textwidth]{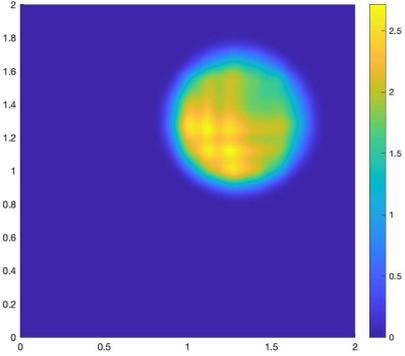}
    \caption{ Final distribution $\ov {\mathsf m}^\Delta(T,\cdot)$ for $\gamma=0.5$.}
    \label{dist_K_1}
\end{subfigure}
\vfill
\vspace{0.6cm}
\begin{subfigure}{0.45\textwidth}
    \centering
    \includegraphics[width=\textwidth]{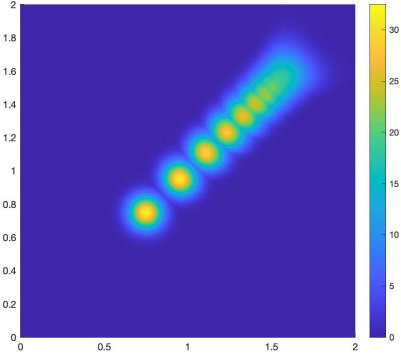}
    \caption{Time evolution $(\ov {\mathsf m}^\Delta(t_k,\cdot))_{k\in \I_{\Dt}}$   for $\gamma=3$.}
    \label{mass_K_2}
\end{subfigure}
\hfill
\label{figures}
\begin{subfigure}{0.45\textwidth}
    \centering
    \includegraphics[width=\textwidth]{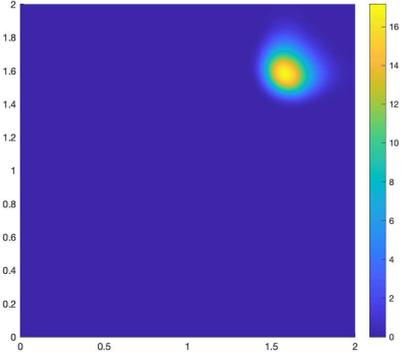}
    \caption{Final distribution $\ov {\mathsf m}^\Delta(T,\cdot)$ for $\gamma=3$.}
    \label{dist_K_2}
\end{subfigure}
\caption{Approximation of $m^*$ in both cases $\gamma=0.5$ and $\gamma=3$.}
\end{figure}

\appendix
\section{}
\label{appendix}
\begin{proof}[Proof of Proposition~\ref{prop:value_function}]   Let us fix $(t,x)\in [0,T[\times\RR^{d}$.  The existence of $\alpha^{t,x}\in L^{2}\big([t,T];\RR^{d}\big)$, such that $v[\mu](t,x)=J^{t,x}[\mu](\alpha^{t,x})$, follows from \eqref{h:L_bounded_below_quadratic_term}, the continuity assumption on $F$ and $G$ in~{\bf(H\ref{h:h2})}, and the direct method in the calculus of variations. Setting 
$\alpha_{0}(s)=0$ for all $s\in [t,T]$, the inequalities $J^{t,x}[\mu](\alpha^{t,x})\leq J^{t,x}[\mu](\alpha_{0}(s))$,~\eqref{h:L_bounded_above_quadratic_term}, ~\eqref{h:F_bounded},~\eqref{h:G_bounded}, and~\eqref{h:L_bounded_below_quadratic_term}, imply that 
\be 
\label{e:uniform_l2_bound_above}
\int_{t}^{T}|\alpha^{t,x}(s)|^2\dd s\leq \tilde{C}:=\frac{T(C_{L,2}+2C_{F,1}+C_{L,7})+2C_{G,1}}{C_{L,6}}.
\ee
In particular, setting $\A^{t}=\Big\{\alpha\in L^{2}\big([t,T];\RR^{d}\big),\;\int_{t}^{T}|\alpha(s)|^2\dd s\leq \tilde{C}\Big\}$, we have 
\begin{equation}
\label{def:v_mu_bounded_integral}
v[\mu](t,x)=\inf\Big\{J^{t,x}[\mu](\alpha)\,\big|\,\alpha\in\A^{t}\Big\}.
\end{equation}
Thus, assertion~\eqref{prop:value_function_i_bis} follows from ~\eqref{h:L_bounded_above_quadratic_term},~\eqref{h:L_bounded_below_quadratic_term},
~\eqref{e:uniform_l2_bound_above},~\eqref{h:F_bounded}, and~\eqref{h:G_bounded}. Moreover, it follows from conditions~\eqref{h:L_Lipschitz},~\eqref{h:F_Lipschitz},~\eqref{h:G_Lipschitz}, and expression~\eqref{def:v_mu_bounded_integral} that, for every $y\in\RR^d$, we have
\begin{align}
|v[\mu](t,x)-v[\mu](t,y)|&\leq\sup_{\alpha\in\A^{t}}\big|J^{t,x}[\mu](\alpha)-J^{t,y}[\mu](\alpha)\big|\nonumber\\
&\leq\sup_{\alpha\in\A^{t}}\Bigg\{\int_{t}^{T}\Big(C_{L,3}(1+|\alpha(s)|^2)+C_{F,2}\Big)|X^{t,x,\alpha}(s)-X^{t,y,\alpha}(s)|\dd s\nonumber\\
&\hspace{6.3cm}+C_{G,2}|X^{t,x,\alpha}(T)-X^{t,y,\alpha}(T)|\Bigg\}\nonumber\\
&\leq\Big(T(C_{L,3}+C_{F,2})+C_{L,3}\tilde{C}+C_{G,2}\Big)|x-y|,
\nonumber
\end{align}
which shows~\eqref{prop:value_function_ii}. Let us set $\ov{X}=X^{t,x,\alpha^{t,x}}$ and let $s\in[t,T[$. Since $v[\mu]$ satisfies the dynamic programming inequality
\begin{multline*}
v[\mu](s,\ov{X}(s))\leq\int_{s}^{s+h}\Big(L(X^{s,\ov{X}(s),\alpha}(r),\alpha(r))+F(X^{s,\ov{X}(s),\alpha}(r),\mu(r))\Big)\dd r+v[\mu]\big(s+h,X^{s,\ov{X}(s),\alpha}(s+h)\big),
\end{multline*}
for all $h\in [0,T-s[$ and $\alpha\in L^{2}([t,T];\RR^d)$, by taking $\alpha=\alpha_0$, the equality 
\be
v[\mu](s,\ov{X}(s))=J^{s,\ov{X}(s)}[\mu](\alpha^{t,x}|_{[s,T]}),
\ee
the estimates~\eqref{h:L_bounded_below_quadratic_term},~\eqref{h:L_bounded_above_quadratic_term},
~\eqref{h:F_bounded}, the equality $X^{s,\ov{X}(s),\alpha_{0}}(s+h)=\ov{X}(s)$, and~\eqref{eq:value_function_Lipschitz}, imply that
\begin{align}
C_{L,6}\int_{s}^{s+h}|\alpha^{t,x}(r)|^2\dd r&\leq h(C_{F,1}+C_{L,2})+h(C_{L,7}+C_{F,1})\nonumber\\
&\hspace{0.3cm}+v[\mu](s+h,\ov{X}(s))-v[\mu](s+h,\ov{X}(s+h))\nonumber\\
&\leq h(2C_{F,1}+C_{L,2}+C_{L,7})+C_{\text{{\rm Lip}}}\int_{s}^{s+h}|\alpha^{t,x}(r)|\dd r.
\nonumber
\end{align}
By Young's inequality, we get the existence of $C>0$, independent of $(\mu,t,x)$, such that 
$$
\int_{s}^{s+h}|\alpha^{t,x}(r)|^2\dd r\leq C h
$$
and, hence, by the Lebesgue differentiation theorem (see e.g.~\cite{MR2267655}), we have $\alpha^{t,x}\in L^{\infty}([0,T];\RR^d)$ and $\|\alpha^{t,x}\|_{L^{\infty}([0,T];\RR^d)}\leq \sqrt{C}$, which shows~\eqref{prop:value_function_i}.

Finally, in order to show~\eqref{prop:value_function_iii}, notice that, for every $y\in\RR^d$,~\eqref{h:L_second_order_derivative_x_bounded_above} implies that
\begin{equation}
\label{e:L_semiconcavity_alpha_square}
L(x+y,a)-2L(x,a)+L(x-y,a)\leq C_{L,5}(1+|a|^2)|y|^2\quad\text{for all }a\in\RR^d. 
\end{equation} 
Estimates~\eqref{e:L_semiconcavity_alpha_square},~\eqref{h:F_semiconcave},
~\eqref{h:G_semiconcave}, and~\eqref{e:uniform_l2_bound_above}, imply
\begin{align}
v(t,x+y)+v(t,x-y)&\leq \int_{t}^{T}\Big(L(X^{t,x+y,\alpha^{t,x}}(s),\alpha^{t,x}(s))+L(X^{t,x-y,\alpha^{t,x}}(s),\alpha^{t,x}(s))\nonumber\\
&\hspace{0.3cm}+F(X^{t,x+y,\alpha^{t,x}}(s),\mu(s))+F(X^{t,x-y,\alpha^{t,x}}(s),\mu(s))\Big)\dd s\nonumber\\
&\hspace{0.3cm}+G(X^{t,x+y,\alpha^{t,x}}(T),\mu(T))+G(X^{t,x-y,\alpha^{t,x}}(T),\mu(T))\nonumber\\
&\leq 2\int_{t}^{T}\Big(L(X^{t,x,\alpha^{t,x}}(s),\alpha^{t,x}(s))+F(X^{t,x,\alpha^{t,x}}(s),\mu(s))\Big)\dd s\\
&\hspace{0.3cm}+2 G(X^{t,x,\alpha^{t,x}}(T),\mu(T))\nonumber\\
&\hspace{0.3cm}+\int_{t}^{T}\Big(C_{L,5}(1+|\alpha^{t,x}(s)|^2)+C_{F,3}\Big)|y|^2\dd s+C_{G,3}|y|^2\nonumber\\
&\leq 2v[\mu](t,x)+\Big(T(C_{L,5}+C_{F,3})+C_{L,5}\tilde{C}+C_{G,3}\Big)|y|^2, \nonumber
\end{align}
from which~\eqref{eq:value_function_sc} follows. 
\end{proof}
\bibliographystyle{plain}
\bibliography{bibFP}
\end{document}